\documentclass[12pt]{article}

\usepackage{amsmath,amssymb,latexsym, amsbsy, amsfonts, amscd, amsthm, mathrsfs, xy,comment, xcolor}
\usepackage{multirow}
\usepackage{hyperref}
\usepackage{enumitem}
 \usepackage[bibtex-style]{amsrefs}
\xyoption{all}

\theoremstyle{plain}

\newtheorem{ithm}{Theorem}

\newtheorem{theorem}{Theorem}[section]

\newtheorem{prop}[theorem]{Proposition}
\newtheorem{corollary}[theorem]{Corollary}

\newtheorem{lemma}[theorem]{Lemma}

\newenvironment{psmallmatrix}
  {\left(\begin{smallmatrix}}
  {\end{smallmatrix}\right)}

\theoremstyle{definition}
\newtheorem{definition}[theorem]{Definition}
\newtheorem{remark}[theorem]{Remark}

\long\def\symbolfootnote[#1]#2{\begingroup
\def\thefootnote{\fnsymbol{footnote}}\footnote[#1]{#2}\endgroup}

\def\GL{{\bf GL}}

\def\PP{{\mathbf P}}

\def\cA{\mathcal{A}}

\def\sgn{\mathrm{sgn}}
\def\N{\mathrm{N}}
\def\1{\mf{1}}

\DeclareMathOperator{\cond}{cond}

 \DeclareMathOperator{\End}{End}

\DeclareMathOperator{\sign}{sign}
 
\DeclareMathOperator{\ord}{ord}
\def\dord{\text{-}\ord{}}

\DeclareMathOperator{\cusps}{cusps}

 \DeclareMathOperator{\Cl}{Cl}

\DeclareMathOperator{\con}{con}

\newcommand{\mat}[4]{\left( \begin{array}{cc} {#1} & {#2} \\ {#3} & {#4}
\end{array} \right)}

\newcommand{\stack}[2]{\genfrac{}{}{0pt}{}{#1}{#2}}

\newcommand{\mf}{\mathfrak }

\def\fa{\mathfrak{a}}
\def\fn{\mathfrak{n}}
\def\fp{\mathfrak{p}}
\def\fq{\mathfrak{q}}
\def\fg{\mathfrak{g}}
\def\fm{\mathfrak{m}}
\def\ft{\mathfrak{t}}

\def\fr{\mathfrak{r}}
\def\fs{\mathfrak{s}}

\def\fl{\mathfrak{l}}
\def\fP{\mathfrak{P}}

\def\fd{\mathfrak{d}}
\def\fQ{\mathfrak{Q}}

\def\T{\mathbf{T}}
\def\Z{\mathbf{Z}}

\def\Q{\mathbf{Q}}
\def\C{\mathbf{C}}

\def\R{\mathbf{R}}

\def\bdf{\begin{defn}}
\def\edf{\end{defn}}

\def\cH{\mathcal{H}}

\def\cO{\mathcal{O}}

\def\cP{\mathcal{P}}
\def\cU{\mathcal{U}}

\def\cR{\mathcal{R}}

\def\ff{\mathfrak{f}}
\def\fb{\mathfrak{b}}
\def\fc{\mathfrak{c}}

\def\cP{{\cal P}}

\begin{document}

\baselineskip=17pt

\title{On Constant terms of Eisenstein Series}
\author{Samit Dasgupta \and Mahesh Kakde}

\date{}

\maketitle

\renewcommand{\thefootnote}{}

\footnote{2020 \emph{Mathematics Subject Classification}: Primary 11F41; Secondary 11F30.}

\footnote{\emph{Key words and phrases}: Eisenstein series, Hilbert modular forms, Ordinary forms.}

\renewcommand{\thefootnote}{\arabic{footnote}}
\setcounter{footnote}{0}

\begin{abstract}
We calculate the constant terms of certain Hilbert modular Eisenstein series at all cusps.  Our formula relates these constant terms to special values of Hecke $L$-series.  This builds on previous work of Ozawa, in which a restricted class of Eisenstein series were studied.  Our results have direct arithmetic applications---in separate work we apply these formulas to prove the Brumer--Stark conjecture away from $p=2$ and to give an exact analytic formula for Brumer--Stark units.
\end{abstract}

\tableofcontents

\section{Introduction}

Let $F$ be a totally real field of degree $d$, and let $M_k(\fn)$ denote the space of Hilbert modular forms of level $\fn \subset \cO_F$ and weight $k$ over $F$.  Let $E_k(\fn) \subset M_k(\fn)$ denote the subspace of Eisenstein series.
In this paper we generalize results of \cite{ddp}*{Section 2.1} and \cite{ozawa} to give the constant terms of nearly all Eisenstein series 
$E \in E_k(\fn)$ at all cusps.  The space $E_k(\fn)$ has a basis consisting of forms of the form $E_k(\eta, \psi)|_\fm$, where $\eta$ and $\psi$ are primitive ray class characters (see \S\ref{s:eisenstein}).  Our formula in Theorem~\ref{t:levelraised} gives the constant terms of these series at all cusps when $\fm$ is squarefree and coprime to the conductors of $\eta$ and $\psi$.  In fact, Theorems~\ref{t:eiscon1} and~\ref{t:levelraised} are more general than this; in particular we handle the case where $\eta$ and $\psi$ are not necessarily primitive characters.  We  work with all weights $k \ge 1$. In \cite{ozawa}, only primitive characters are considered, the level raising operator $|_\fm$ is not applied, and the weight $k$ taken to be at least $2$. 

There are concrete arithmetic applications of our results.  In \cite{dk}, we prove the Brumer--Stark conjecture away from $p=2$ and in \cite{dk2} we prove an exact $p$-adic formula for Brumer--Stark units.  Broadly speaking, both of these results apply Ribet's method, whereby cusp forms are constructed by taking linear combinations of products of Eisenstein series \cite{ribet}.  Central to the advance of \cite{dk} is the method by which this cusp form is constructed.  For this, we require knowledge of the constant terms at all cusps of level-raised Eisenstein series associated to possibly imprimitive characters; we also need to include weight $k=1$.   Therefore the calculations of \cite{ozawa} are not general enough for our application, which provides the motivation for this paper.  

In addition, we prove here some other results that may be of independent interest.  Firstly, we provide a complete enumeration of the cusps on the Hilbert modular variety.
Also, we prove that in weight $k > 1$, the cuspidality of modular forms that are ordinary at a prime $p$ is regulated by the constant terms at cusps that are unramified at $p$.  We provide two proofs of this fact; one applies our results on Eisenstein series, and the other is a direct study of the $U_p$ operator.  While these  two results are likely known to the experts, we have not found a precise reference for them in the literature.

\bigskip

We now outline the paper and describe our results in greater detail.
In \S\ref{s:hmf} we recall the definition of the space of Hilbert modular forms $M_k(\fn)$ of weight $k$ and level $\fn \subset \cO_F$, following 
Shimura \cite{shim}.   Associated to each $\lambda$ in the narrow class group $\Cl^+(F)$  is a congruence subgroup
 $\Gamma_{1, \lambda}(\fn) \subset \GL_2^+(F)$.  The open Hilbert modular variety corresponding to our forms has $h^+ = \#\Cl^+(F)$ 
 components:

\[ 
Y = \bigsqcup_{\lambda \in \Cl^+(F)} \Gamma_{1, \lambda}(\fn) \backslash \cH^d, \qquad \cH = \text{complex upper half plane.} 
\]

The space of modular forms $M_k(\fn)$ is endowed with an action of Hecke operators described in \S\ref{s:heckedefs}. 
Among these operators are the diamond operators $S(\fm)$, indexed by the classes $\fm \in G_\fn^+$, the narrow ray class group of $F$ attached to the modulus $\fn$.  The diamond operators play a central role in our applications \cite{dk}, \cite{dk2}.

In \S\ref{s:cusps}, we study the set of cusps associated to $Y$:
\[ \cusps(\fn) = \bigsqcup_{\lambda \in \Cl^+(F)} \Gamma_{1, \lambda}(\fn) \backslash \PP^1(F). \]
We provide an explicit enumeration of this set.  For $\fm \mid \fn$, let $Q_{\fm, \fn}$ denote the quotient of $G_{\fm}^+ \times G_{\fn/\fm}^+$ by the subgroup generated by diagonally embedded principal ideals $(x)$, where $x \in \cO_F$ is congruent to $1$ modulo $\fn$.   The following result proved in \S\ref{s:enumerate} is already implicit in \cite{wileseis}*{Pp.\ 422-423}. 

\begin{ithm}  There is a stratification $\cusps(\fn) = \bigsqcup_{\fm \mid \fn} \fQ_{\fm, \fn}$ with $\#\fQ_{\fm, \fn} = \#Q_{\fm, \fn}$.  Each $\fQ_{\fm, \fn}$ is stable under the action of $G_\fn^+$ via the diamond operators.
\end{ithm}

In \S\ref{s:eisenstein} we study the Eisenstein series in $M_k(\fn)$ and calculate their constant terms at all cusps. This generalizes the results of \cite{ddp}*{Proposition 2.1} and \cite{ozawa}. We work in a more general setting in this paper by considering all cusps and allowing for Eisenstein series associated to imprimitive characters.  We normalize our constant terms  (see (\ref{e:cadef}) below) so that they are independent of  choice of representatives (up to sign).  Furthermore, with these normalizations the constant terms exhibit nice integrality properties that are studied by Silliman in \cite{dks}. For an ideal $\fb \mid \fn$, define
\begin{equation} \label{e:cibn}
 C_\infty(\fb, \fn) = \bigsqcup_{\fb \mid \fm} \fQ_{\fm, \fn}. \end{equation}

In Theorem~\ref{t:eiscon1} we prove the following.

\begin{ithm}
Let $k >  1$, and let $\chi_1$ and $\chi_2$ be narrow ray class characters of $F$ with associated signs $q_1, q_2 \in (\Z/2\Z)^n$, respectively.  Assume that $\chi_2$ is primitive of conductor $\fb$. Then the constant term of $E_k(\chi_1, \chi_2)$ vanishes 
at any cusp not lying in $C_\infty(\fb, \fn)$.
Furthermore, if \[ \cA \in \Gamma_{1,\lambda}(\fn) \backslash \PP^1(F) \] is represented by $a/c \in \PP^1(F)$ and lies in 
$C_\infty(\fb, \fn),$
 the constant term of $E_k(\chi_1, \chi_2)$ at $\cA$ normalized as in (\ref{e:cadef}) is given by
\begin{equation} \label{e:constterm0}
\begin{aligned}
\frac{1}{2^n} \frac{\tau(\chi_1 \chi_2^{-1})}{\tau(\chi_2^{-1})} \left( \frac{\N\fb}{\N\ff}\right)^k & \sgn(-c)^{q_1} \sgn(a)^{q_2} \chi_1(\fc_{\cA}/\fb) \chi_2^{-1}(\fa_{\cA}) \\
& \times  L(\chi^{-1}, 1-k) \prod_{\fq}(1 - \chi(\fq) \N\fq^{-k}).
\end{aligned}
\end{equation}
Here $\chi$ denotes the  primitive character associated to $\chi_1\chi_2^{-1}$, $\ff = \cond(\chi)$, and 
  $\fq$ runs through all primes dividing $\fn$ but not $\ff$.  The integral ideals $\fa_\cA$ and $\fc_\cA$ associated to $\cA$ are defined in (\ref{e:acdef}), and the condition $\cA \in C_\infty(\fb, \fn)$ implies that $\fb \mid \fc_\cA$.
\end{ithm}

  In Theorem~\ref{t:eiscon1} we also consider the  case $k=1$.  In Theorem~\ref{t:levelraised} we build on the result above and consider a more generalize case; we calculate the constant terms of all level-raised Eisenstein series $E_k(\chi_1, \chi_2)|_\fm$, where $\chi_1$ and $\chi_2$ are possibly imprimitive, under certain mild conditions.
These results are essential in our arithmetic applications \cite{dk} and \cite{dk2}.  In those works, we construct cusp forms by taking the appropriate linear combinations of products of  Eisenstein series considered here with certain other auxiliary forms constructed in \cite{dks}.

In \S\ref{s:ordinary} we conclude with the following  result on the cuspidality of ordinary forms that is applied in our arithmetic applications \cite{dk}, \cite{dk2}.  Fix a prime $p$ and let 
 $\fP = \gcd(p^\infty, \fn)$ denote the $p$-part of $\fn$.  The set $C_\infty(\fP, \fn)$ defined in (\ref{e:cibn}) may be viewed as the set of ``$p$-unramified" cusps. 
\begin{ithm} \label{t:ord} Let $p$ be a prime.  If $f \in M_k(\fn)$ is $\fp$-ordinary for each prime $\fp \subset \cO_F$ dividing $p$, then $f$ is cuspidal if and only if the constant term of $f$ vanishes at each cusp in $C_\infty(\fP, \fn)$.
\end{ithm}

\section{Notation on Hilbert Modular Forms} \label{s:hmf}

We refer the reader to \cite[\S2.1]{ddp} for our precise definitions and notations, following Shimura \cite{shim}, concerning the space of classical Hilbert modular forms over the totally real field $F$.  We recall certain aspects of this definition.

\subsection{Hilbert Modular Forms}

Let $\cH$ denote the complex upper half plane endowed with the usual action of $\GL_2^+(\R)$ via linear fractional transformations, where $\GL_2^+$ denotes the group of matrices with positive determinant.  We fix an ordering of the $n$ embeddings $F \hookrightarrow \R$, which yields an embedding of $\GL_2^+(F) \hookrightarrow \GL_2^+(\R)^n$ and hence an action of $\GL_2^+(F)$ on $\cH^n$.  Here $\GL_2^+(F)$ denotes the group of matrices with totally positive determinant.

For each class $\lambda$ in the narrow class group $\Cl^+(F)$, we choose a representative fractional ideal $\ft_\lambda$.  Let $\fn \subset \cO_F$ be an ideal, and assume that the representative ideals $\ft_\lambda$ have been chosen to be relatively prime to $\fn$.  Define the groups
\begin{align*}
 \Gamma_{0, \lambda}(\fn) &= \left\{ \mat{a}{b}{c}{d} \in \GL_2^+(F): a, d \in \cO_F, c \in \ft_\lambda \fd \fn, b \in (\ft_\lambda \fd)^{-1}, ad-bc \in \cO_F^* \right\}, \\ 
  \Gamma_{1, \lambda}(\fn) &= \left\{ \mat{a}{b}{c}{d} \in \Gamma_{0, \lambda}: d \equiv 1 \pmod{\fn} \right\}.
 \end{align*} 
Here $\fd$ denotes the different of $F$.  

Let $k$ be a positive integer.  We denote by $M_k(\fn)$ the space of Hilbert modular forms for $F$ of level $\fn$ and weight $k$.  Each element $f \in M_k(\fn)$ is a tuple $f = (f_\lambda)_{\lambda \in \Cl^+(F)}$ of holomorphic functions $f_\lambda\colon \cH^n \rightarrow \C$ 
such that $f_\lambda|_{\alpha,k } = f_\lambda$ for all $\lambda \in \Cl^+(F)$ and $\alpha \in \Gamma_{1, \lambda}$.  Here the weight $k$ slash action is defined in the usual way:
\[ f_\lambda|_{\alpha,k }(z_1, \dotsc, z_n) = \N(\det(\alpha))^{k/2} \prod_{i=1}^n (c_i z_i + d_i)^{-k} \cdot f_\lambda\left( \frac{a_1 z_1 + b_1}{c_1 z_1 + d_1}, \dotsc, \frac{a_n z_n + b_n}{c_n z_n + d_n}\right), \]
where $a_i$ denotes the image of $a$ under the $i$th real embedding of $F$ and similarly for $b_i, c_i, d_i$.

\subsection{Constant terms and cusp forms} \label{s:constant}

  Suppose that $\cA = (A, \lambda)$ is an ordered pair with  \[ A= \mat{a}{b}{c}{d} \in \GL^+(F)\] and $\lambda \in \Cl^+(F)$.  We define the fractional ideal
 \[ \fb_\cA = a \cO_F + c(\ft_\lambda \fd)^{-1}. \]

Given $f = (f_\lambda) \in M_k(\fn)$ and a pair $\cA = (A, \lambda)$ as above, the function $f_\lambda|_{A, k}$ has a Fourier expansion
\begin{equation} \label{e:qexp}
f_\lambda|_{A, k}(z) = a_{\cA}(0) + \sum_{\stack{b \in \fa}{b \gg 0}} a_{\cA}(b) e_F(bz), 
\end{equation}
where $\fa$ is a lattice in $F$ depending on $\cA$, and
\[ 
e_F(bz) := \exp(2 \pi i (b_1 z_1 + \cdots +  b_n z_n)). 
\]

\begin{definition} The {\em normalized constant term} of the form $f$ at $\cA$ is
\begin{equation} \label{e:cadef}
c_{\cA}(0,f) = a_{\cA}(0) \cdot (\N\ft_{\lambda})^{-k/2}(\N\fb_\cA)^{-k}(\det A)^{k/2}.
\end{equation}
\end{definition}
As we will see later, the constant terms with this normalization will exhibit nice invariance properties as well as integrality properties. 
The space of {\em cusp forms} $S_k(\fn) \subset M_k(\fn)$ is defined to be subspace of forms $f$ such that $c_{\cA}(0, f)=0$ for all pairs $\cA$.

\subsection{$q$-expansion}  \label{s:q}

When $A=1$ we drop the subscript $A$ and write simply 
\[
c_{\lambda}(0, f) = a_{\lambda}(0)  (\N\ft_{\lambda})^{-k/2}.
\]

Furthermore when $A=1$, the lattice $\fa$ appearing in (\ref{e:qexp}) is the ideal $\ft_{\lambda}$. Any non-zero integral ideal $\fm$ may be written  $\fm = (b) \ft_{\lambda}^{-1}$ with $b \in \ft_{\lambda}$ totally positive for a unique $\lambda \in \Cl^+(F)$. We define the {\em normalized Fourier coefficients}
\[
c(\fm, f) = a_{\lambda}(b) (\N\ft_{\lambda})^{-k/2}.
\]
The collection of normalized Fourier coefficients $\{c_{\lambda}(0,f), c(\fm, f)\}$ is called the \emph{$q$-expansion} of $f$.
Note that these normalized coefficients are denoted with a capital $C$ in \cite{shim}.
\subsection{Forms over a field $K$} \label{s:formsoverr}

Each tuple $f \in M_k(\fn)$ is determined by its $q$-expansion, which the collection of coefficients
\[ 
c_\lambda(0, f) \in \C, \lambda \in \Cl^+(F), \qquad c(\fm, f) \in \C, \fm \subset \cO_F, \fm \neq 0
\]
defined in \S\ref{s:q}.
For any subfield $K \subset \mathbf{C}$, define $M_k(\fn, K)$ to be the $K$-vector subspace of $M_k(\fn)$ consisting of modular forms whose $q$-expansion coefficients all lie in $K$. A fundamental result of Shimura \cite[Theorem~7]{shimura} states
\begin{equation}
\label{e:mkkdef}
M_k(\fn, K) = M_k(\fn, \Q) \otimes_{\Q} K.
\end{equation}
We define $M_k(\fn, K)$ by (\ref{e:mkkdef}) more generally if $K$ is any field of characteristic 0.
This generalizes in the obvious way to define $S_k(\fn, K)$.

\subsection{Hecke operators} \label{s:heckedefs}

The space $M_k(\fn)$ is endowed with the action of a Hecke algebra $\tilde{\T} \subset \End(M_k(\fn))$ generated over $\Z$ by the following operators:
\begin{itemize}
\item $T_\fq$ for $\fq \nmid \fn$.
\item $U_\fq$ for $\fq \mid \fn$.
\item The ``diamond operators" $S(\fm)$ for  each class $\fm \in G_\fn^+ = $ narrow ray class group of $F$ of conductor $\fn$.
\end{itemize}

We refer to \cite[\S2]{shim} for the definition of these operators. We warn that in \emph{loc.\ cit.}\  both $T_{\fq}$ and $U_{\fq}$ are denoted by $T_{\fq}$. 

Let us recall the definition of the diamond operators $S(\fm)$.  Let $f = (f_\lambda)_{\lambda \in \Cl^+(F)} \in M_k(\fn)$ and let $\fm$ denote an  ideal of $\cO_F$ that is relatively prime to $\fn$, and which hence represents a class in $G_\fn^+$.  For each $\mu \in \Cl^+(F)$, let $\lambda \in \Cl^+(F)$ denote the class of $\mu \fm^{-2}$.  Write $\ft_\lambda \ft_\mu^{-1} \fm^2 = (x)$ where $x$ is a totally positive element of $F^*$, uniquely determined up to multiplication by a totally positive unit in $\cO_F^*$.  Let \begin{equation} \label{e:alphamu}
\alpha_\mu = \mat{a}{b}{c}{d} \in \GL_2^+(F) \end{equation} be a matrix satisfying the following conditions: \begin{equation} 
\label{e:alphamudef}
 a \in \fm,\ \ b \in \ft_\mu^{-1} \fd^{-1} \fm, \ \ c \in \ft_\lambda \fd \fn \fm, \ \ d \in \ft_\lambda \ft_\mu^{-1} \fm, \ \ \det(\alpha_\mu) = x, \ \ d \equiv x \pmod{\ft_\lambda \ft_\mu^{-1} \fm \fn}. \end{equation}
Then \begin{equation} \label{e:smdef}
f|_{S(\fm)} = (g_\mu)_{\mu \in \Cl^+(F)} \quad  \text{ where  } g_\mu = f_\lambda|_{\alpha_\mu}. \end{equation}

\subsection{Raising the level} \label{s:level}

For a Hilbert modular form $f \in M_k(\fn)$ and an integral ideal $\fq$ of $F$, there is a form \[ f|\fq \in M_k(\fn \fq) \] characterized by the fact that for nonzero integral ideals $\fa$ we have
\begin{equation} \label{e:fq}
 c(\fa, f|\fq) = \begin{cases}
c(\fa/\fq, f) & \text{if } \fq \mid \fa \\
0 & \text{if } \fq \nmid \fa
\end{cases} \end{equation}
and 
\begin{equation} \label{eq:levelconstant} 
c_\lambda(0, f|\fq) = c_{\lambda \fq}(0, f)
\end{equation}
for all $\lambda \in \Cl^+(F)$.  We recall the definition of $f|\fq$.  For every $\lambda$ there is a $\mu \in \Cl^+(F)$ and a totally positive element $a_{\mu} \in F$ such that
\[
\fq \ft_{\lambda} = (a_{\mu}) \ft_{\mu}.
\]
Then \begin{equation} \label{e:lrdef}
 (f|\fq)_{\lambda} := \N\fq^{-k/2} f_{\mu}|_{\left( \begin{array}{cc} a_{\mu} & 0 \\ 0 & 1\end{array}\right)}. \end{equation}
The fact that $f|\fq$ lies in $M_k(\fn\fq)$ and satisfies (\ref{e:fq})--(\ref{eq:levelconstant}) is proven in \cite[Prop 2.3]{shim}.

\section{Cusps}  \label{s:cusps}

\subsection{Admissibility} \label{ss:adm}

  Recall  the fractional ideal
 \begin{equation} \label{e:bdef}
 \fb_\cA = a \cO_F + c(\ft_\lambda \fd)^{-1} \end{equation}
 defined in \S\ref{s:constant} associated to a pair $\cA = (A, \lambda)$ with  $A \in \GL^+(F)$ and $\lambda \in \Cl^+(F)$.
We now define the integral ideals  \begin{equation} \label{e:acdef}
  \fa_\cA = a \fb_\cA^{-1}, \qquad \fc_\cA = c (\ft_\lambda \fd \fb_\cA)^{-1}. \end{equation}
   The ideals $\fa_\cA, \fc_\cA$ are relatively prime.   
   
   To explain the meaning of these invariants, consider the case $F=\Q$.  If $\cA$ represents the cusp $a/c \in \PP^1(\Q)$ with $a, c$ relatively prime integers, then $\fa_\cA = (a)$ and $\fc_\cA = (c)$.  Finally, we define
 \begin{equation} \label{e:mdef}
  \fm_\cA = \gcd(\fc_\cA, \fn). \end{equation}
   
Given a form $f \in M_k(\fn)$, 
it is clear that the normalized constant term $c_{\cA}(0, f)$ defined in \S\ref{s:constant} depends only on $A$ up to left multiplication by an element of $\Gamma_{1, \lambda}(\fn)$.
Furthermore, writing 
 \[ B = \left\{\mat{a}{b}{0}{d} \in \GL_2^+(F)\right\}, \]
it is {\em almost} true that $c_{\cA}(0, f)$ depends  on $A$ up to right multiplication by an element of $B$---there is a sign ambiguity 
\begin{equation} \label{e:borel}
 c_{(AA', \lambda)}(0, f) = \sign(\N d)^k c_{(A, \lambda)}(0, f), \qquad A'= \mat{a}{b}{0}{d} \in B. 
 \end{equation}
If $k$ is odd, and the class of $A$ in $\Gamma_{1, \lambda}(\fn) \backslash \GL_2^+(F)$
is fixed under right multiplication by an element $A' \in B$ with $\N d < 0$, then it follows from (\ref{e:borel}) that
$c_{\cA}(0, f) = 0$.   Let us determine the pairs $\cA$ for which this is the case.  
  
\begin{definition}
A pair $(\fm, \fn)$ with $\fm \mid \fn$ is called {\em admissible} if there does not exist a pair of units $\epsilon_1, \epsilon_2 \in \cO_F^*$ such that $\epsilon_1 \equiv 1 \pmod{\fm}, \epsilon_2 \equiv 1 \pmod{\fn/\fm}$, with 
$\N\epsilon_1 = \N\epsilon_2 = -1$ and $\epsilon_1/\epsilon_2$ totally positive. 
\end{definition}

\begin{definition}  With the level $\fn$ fixed,  a pair $\cA = (A, \lambda)$ 
with $A \in \GL_2^+(F)$ and $\lambda \in \Cl^+(F)$ is called admissible if $(\fm_A, \fn)$ is admissible.
\end{definition}

\begin{theorem}  \label{t:admissible}
Given a pair $\cA = (A, \lambda)$, the class of $A$ in $\Gamma_{1, \lambda}(\fn) \backslash \GL_2^+(F)$ is fixed by right multiplication by an element $A' \in B$ with $\N d < 0$ if and only if $\cA$ is not admissible.
\end{theorem}

Before proving the theorem, we introduce some notation and prove an important lemma.  Given a fractional ideal $\fb$ and an integral ideal $\fm$, we denote by $(\fb / \fb \fm)^*$ the subset of elements of $\fb/\fb\fm$ that generate this quotient as an $\cO_F/\fm$-module.  This is a principal homogeneous space for the group $(\cO_F/\fm)^*$.  

\begin{definition} For a fractional ideal $\fb$ and an integral ideal $\fm$, we define 
\[ \cR^{\fb}_{\fm}  = (\fb / \fb \fm)^*/ \cO_{F,+}^{*}, \]
the quotient of the set $(\fb / \fb \fm)^*$ by the action of multiplication by the group of totally positive units of $F$.
\end{definition}

\begin{definition} \label{d:pl}
Let  $\cP_{\lambda}(\fn)$  be the set of tuples $(\fb, \fm, a, c)$ where $\fb$ is a fractional ideal of $F$, $\fm$ is an integral ideal dividing $\fn$, $a \in \cR^{\fb}_{\fm}$, and $c \in \cR^{\fb \fd \ft_\lambda \fm}_{\fn/\fm}$.
\end{definition}

  The heart of Theorem~\ref{t:admissible} is the following lemma.

\begin{lemma}  \label{l:bijection}
Fix $\lambda \in \Cl^+(F)$. 
  
  There is a canonical bijection
\[ \varphi \colon \Gamma_{1, \lambda}(\fn) \backslash ( F^2 \setminus (0,0)) \longrightarrow \cP_{\lambda}(\fn) \]
given by $(a, c) \mapsto (\fb, \fm, \overline{a}, \overline{c}), $ where $\fb = \fb_\cA$ and $\fm = \fm_\cA$ are defined as in (\ref{e:bdef}) and (\ref{e:mdef}).
\end{lemma}

\begin{proof}  The fact that the map $\varphi$ is well-defined is elementary and left to the reader.  Surjectivity is also not difficult.
Given a fractional ideal $\fb$ and an integral ideal $\fm \mid \fn$, choose $c \in \ft_\lambda \fd \fb$ such that $\gcd(c \fd^{-1} \ft_\lambda^{-1}, \fn) = \fm$.   Scaling $c$ by an appropriate element of $\cO_F$ relatively prime to $\fn$, we can ensure that $c$ lands in any class in $\cR^{\fb \fd \ft_\lambda \fm}_{\fn/\fm}$ without changing the gcd condition.
 Next we choose any $a \in \cO_F$ such that 
 \begin{equation} \label{e:idealmu} a \cO_F + c (\fd \ft_\lambda)^{-1} = \fb.
 \end{equation}  Scaling $a$ by an appropriate element  of $\cO_F$ relatively prime to $c (\fd \ft_\lambda \fb)^{-1}$, we can ensure that $a$ lands in any class in $\cR_{\fm}^{\fb}$ without affecting (\ref{e:idealmu}).  This proves the desired surjectivity.
 
 For injectivity, suppose that \begin{equation} \label{e:injective}
 \varphi(a, c) = \varphi(a', c') = (\fb, \fm, \overline{a}, \overline{c}).\end{equation}  Let $\cA$ correspond to $(a,c)$ and $\cA'$ to $(a', c')$ as in the statement of the lemma.  From the third component of (\ref{e:injective}), 
  there exists a totally positive unit $\epsilon \in \cO_{F,+}^*$ such that
 $a' \equiv \epsilon a \mod{\fb\fm}$.  
  Since $c \fd^{-1} \ft_\lambda^{-1} + \fb\fn = \fb\fm$,
   we can act by an element of the form 
   $\mat{\epsilon}{b}{0}{1}$ on $(a,c)$ with $b \in (\fd \ft_\lambda)^{-1}$ 
   to ensure that $a \equiv a' \pmod{\fb\fn}$.  
 
 Next let $u \in \cO_{F,+}^*$ such that 
    $c' \equiv u c \pmod{\fb \fd \ft_\lambda \fn}$.  Such a $u$ exists by the fourth component of (\ref{e:injective}).
 Note that the pairs $(a, c), (a', c')$ can be completed to matrices 
 \[  M = \mat{a}{b}{c}{d}\!, \ \det(M) = 1, \qquad M' = \mat{a'}{b'}{c'}{d'}\!, \ \det(M') = u, \]
 with $b, b' \in (\fd \ft_\lambda \fb)^{-1}$ and $d, d' \in \fb^{-1}$.  We claim that $d$ and $d'$ can be chosen to satisfy \[ d' \equiv du \pmod{\fb^{-1} \fn}. \]  
 Granting the claim, it is straightforward to check that \[ M' M^{-1}\binom{a}{c} = \binom{a'}{c'} \qquad \text{ and } \quad M' M^{-1} \in \Gamma_{1, \lambda}(\fn) \] as desired.
 
 It remains prove the claim.
 Given any $y \in \fb^{-1}  \fc_\cA$, we can write $y = cx$ with $x \in (\fd \ft_\lambda \fb^2)^{-1}$ and 
 replace $(b,d)$ by $(b + ax, d + cx)$.  Hence $d$ can be replaced by any element in its equivalence class in $\fb^{-1}/\fb^{-1} \fc_{\cA}$.  Since $\gcd(\fc_\cA, \fn) = \fm$, to prove the claim it therefore suffices to show that $d' \equiv ud \pmod{\fb^{-1} \fm}$.  If $a=0$, this is clear since $\fm =1$.  Otherwise, multiplication by $a$ induces an ismorphism $\fb^{-1}/\fb^{-1}\fm \rightarrow \fa_\cA / \fa_\cA \fm$, so we must show that $ad' \equiv adu \pmod{\fa_\cA \fm}$.  Now $\fm$ divides $\fc_\cA$, which is coprime to $\fa_\cA$, so by the Chinese Remainder Theorem it suffices to separately show that $ad' \equiv adu \pmod{\fa_\cA}$ and $ad' \equiv adu \pmod{\fm}$.  The first of these is trivial since $ad', adu \in \fa_\cA$.  To prove $ad' \equiv adu \pmod{\fm}$ we note
 \[ ad' - adu = b'c' + (a - a')d' - bcu \]
 with $c=b'c' \in \fc_{\cA'} \subset \fm$, $bcu \in \fc_{\cA} \subset \fm$, and $(a-a')d \in \fn \subset \fm$.  This concludes the proof.
\end{proof}

\begin{proof}[Proof of Theorem~\ref{t:admissible}]  Suppose that $\cA = (A, \lambda)$ is not admissible with $A = \mat{a}{*}{c}{*}$, and let $\fm = \fm_\cA$.
Let $\epsilon = \epsilon_1 \in \cO_F^*$ in the definition of admissibility, so $\N\epsilon = -1$, $\epsilon \equiv 1 \pmod{\fm}$, and $\epsilon \equiv \epsilon_1\epsilon_2 \pmod{\fn/\fm}$ with  $\epsilon_1\epsilon_2 \in \cO_F^*$ a totally positive unit.  Right multiplication by $M = \mat{\epsilon}{0}{0}{\epsilon^{-1}}$ sends $A$ to $\mat{a\epsilon}{*}{c\epsilon}{*}$. It is immediate from the definition that $\varphi(a,c) = \varphi(a\epsilon, c\epsilon)$, hence Lemma~\ref{l:bijection} implies that there exists $\gamma \in \Gamma_{1, \lambda}(\fn)$ such that $\gamma A$ and $AM$ have the same first column.   Therefore there exists $N = \mat{1}{*}{0}{*} \in B$ such that $\gamma A = AMN$.  It follows that the class of $A$ in $\Gamma_{1, \lambda}(\fn) \backslash \GL_2^+(F)$ is fixed by right multiplication by $MN \in B$, and the lower right entry of $MN$ has negative norm.

To prove the converse, suppose that $A = \mat{a}{b}{c}{d} \in \GL_2^+(F)$ and the class of $A$ in $\Gamma_{1, \lambda}(\fn) \backslash \GL_2^+(F)$ is fixed by right multiplication by $M = \mat{x}{*}{0}{z} \in B$, where $\N z < 0$.
By Lemma~\ref{l:bijection}, we have $\varphi(a,c) = \varphi(ax, cx)$. From the first coordinate of this equation, we see that $x \in \cO_F^*$.
From the third and fourth components we see that there exist totally positive units $u_1, u_2 \in \cO_{F, +}^*$ such that
\[ x \equiv u_1 \pmod{\fm} \qquad x \equiv u_2 \pmod{\fn/\fm}, \]
where $\fm = \fm_{\cA}$.
Note that $\N x < 0$ (and hence $\N x = -1$) since $\N z < 0$ and $xz$ is totally positive.  Then letting $\epsilon_1 = x/u_1$ and $\epsilon_2 = x/u_2$ shows that $(\fm, \fn)$ is not admissible.
  \end{proof}

\begin{corollary}  Let $k$ be odd, and let $f \in M_k(\fn)$.  The constant term $c_\cA(0, f)$ vanishes if  $\cA$ is not  admissible.
\end{corollary}
\begin{proof} By Theorem~\ref{t:admissible}, if $\cA = (A, \lambda)$ is not admissible then
the class of $A$ in $\Gamma_{1, \lambda}(\fn) \backslash \GL_2^+(F)$ is fixed by right multiplication by an element $A' \in B$ with $\N d < 0$.  Then $c_{\cA}(0, f) = c_{(AA', \lambda)}(0, f)$, but by (\ref{e:borel}) we also have $c_{\cA}(0, f) = -c_{(AA', \lambda)}(0, f)$.  The result follows.
\end{proof}

\subsection{Definition of cusps} 

For $\Gamma \subset \Gamma_{0,\lambda}(\fn)$ any congruence subgroup, the associated set of cusps
 is by definition the finite set
\begin{equation} \label{e:cuspsdef}
\cusps(\Gamma) := \Gamma \backslash \GL_2^+(F) / \left\{ \mat{a}{b}{0}{d} \in \GL_2^+(F) \right\} \leftrightarrow \Gamma \backslash \PP^1(F). 
\end{equation}
The bijection in (\ref{e:cuspsdef}) is $\mat{\alpha}{\beta}{\gamma}{\delta} \rightarrow (\alpha:\gamma)$.  We define 
\begin{equation} \label{e:disjoint} \cusps(\fn) = \bigsqcup_{\lambda} \cusps(\Gamma_{1,\lambda}(\fn)). 
\end{equation}
An ordered pair $\cA = (A, \lambda)$ with $A \in \GL_2^+(F)$ and $\lambda \in \Cl^+(F)$  gives rise to 
an element of $\cusps(\fn)$ by considering the image of $A$ in $\cusps(\Gamma_{1,\lambda}(\fn))$ in the $\lambda$-component of the disjoint union (\ref{e:disjoint}).  The cusp represented by $\cA$ will be denoted $[\cA] \in \cusps(\fn)$.

  \begin{definition}  \label{d:fqdef}
  For each $\fm \mid \fn$, we define $\fQ_{\fm, \fn}$ to be the set of cusps $[\cA]  \in \cusps(\fn)$ such that $\fm_\cA= \fm$. 
  \end{definition}
  
  The set of admissible cusps is defined by
  \[ \cusps^*(\fn) = \bigsqcup_{\substack{\fm \mid \fn \\ (\fm, \fn) \text{ admissible}}}\fQ_{\fm, \fn} = \{ [\cA] \colon \cA \text{ is admissible} \}.\]

There is a canonical action of the diamond operators on $\cusps(\fn)$ that is compatible with its action on modular forms.  Given an integral ideal $\fm$ coprime with $\fn$ and a cusp $\cA = (A, \mu)$, we define $\lambda$ and $\alpha_\mu$ as in (\ref{e:alphamudef}) and define
\[ S(\fm) \cA = \cA' = (A', \lambda), \qquad \text{where} \qquad A' = \alpha_\mu A. \]

\begin{prop}
Each set $\fQ_{\fm, \fn}$ is invariant under the action of $G_\fn^+$ via the diamond operators.
\end{prop}

\begin{proof}  With notation as above, one checks directly from the definitions that $\fb_{\cA'} = \fb_\cA \fm$.  Furthermore one calculates that
\begin{align*}
 \fm_{\cA'} \fb_{\cA'} &= \gcd( (c\alpha + d\gamma)(\fd \ft_\lambda)^{-1} , \fn \fb_{\cA'})  \\
&= \gcd(d\gamma(\fd \ft_\lambda)^{-1} , \fn \fb_{\cA'}) \\
& \subset \gcd(\gamma(\fd \ft_\mu)^{-1} \fm, \fn \fb_{\cA'}) \\
&= \gcd(\fc_\cA \fb_\cA \fm, \fn \fb_{\cA'}) \\
&= \fm_\cA \fb_{\cA'}.
\end{align*}
Therefore $\fm_{\cA'} \subset \fm_\cA$.  However, this is a group action; replacing $\fm$ by an ideal whose image in $G_\fn^+$ is inverse to $\fm$ and switching the roles of $\cA$ and  $\cA'$, we find the reverse inclusion  $\fm_{\cA} \subset \fm_{\cA'}$.  Hence $\fm_{\cA} = \fm_{\cA'}$, and the result follows.
\end{proof}

\subsection{Enumeration of cusps} \label{s:enumerate}

   We now enumerate $\cusps(\fn)$ and more specifically each subset $\fQ_{\fm, \fn}$.
Recall the set of 4-tuples $\cP_{\lambda}(\fn)$ defined in Definition~\ref{d:pl}.
For each ideal $\fm \mid \fn$, let
$\cP_{\lambda}(\fm, \fn) \subset \cP_\lambda(\fn)$ denote the set of tuples whose second coordinate is $\fm$.  There is a natural action of $F^*$ on $\cP_\lambda(\fn)$ that preserves each $\cP_{\lambda}(\fm, \fn)$, given by
\[ x \cdot (\fb, \fm, a, c) = (\fb(x), \fm, ax, cx). \]

The following is an immediate corollary of Lemma~\ref{l:bijection}.
\begin{corollary} \label{c:cusps} For each class $\lambda \in G_1^+$, there is a canonical bijection
\[  \Gamma_{1, \lambda}(\fn) \backslash \PP^1(F)
 \longrightarrow  P_\lambda(\fn)/F^*.
\] 
There are canonical bijections
\begin{align} \label{e:p1p}
 \fQ_{\fm, \fn} \longrightarrow & \bigsqcup_{\lambda \in \Cl^+(F)} P_\lambda(\fm, \fn)/F^*, \\
\cusps(\fn) 
 \longrightarrow &  \bigsqcup_{\lambda \in \Cl^+(F)} P_\lambda(\fn)/F^*.  \nonumber \label{e:cusps}
  \end{align} 
\end{corollary}

\begin{definition}
Let $Q_{\fm, \fn}$ denote the quotient of $G_{\fm}^+ \times G_{\fn/\fm}^+$ by the subgroup generated by diagonally embedded principal ideals $(x)$, where $x \in \cO_F$ is congruent to $1$ modulo $\fn$. 
\end{definition}

\begin{corollary} \label{c:cuspsize}
We have $ \#\fQ_{\fm, \fn} = \#Q_{\fm, \fn}$, hence $\# \cusps(\fn) = \sum_{\fm | \fn} \#Q_{\fm, \fn}.$
\end{corollary}

\begin{proof} 
From (\ref{e:p1p}) we have $\#\fQ_{\fm, \fn} = h^+ \#(P_\lambda(\fm, \fn)/F^*)$, where $h^+ = \#\Cl^+(F)$.
There is a surjective map \begin{equation} \label{e:plproj}
 P_\lambda(\fm, \fn)/F^* \longrightarrow \Cl(F), \qquad (\fb, a, c) \mapsto [\fb]. \end{equation}
If $\cU$ denotes the image of $\cO_F^*$ mapped diagonally to $\cR_{\fm} \times \cR_{\fn/\fm}$, then the fiber over a point in
(\ref{e:plproj}) is a principal homogeneous space for $(\cR_{\fm} \times \cR_{\fn/\fm})/\cU$.  Therefore
\begin{equation} \label{e:qmn}
\#\fQ_{\fm, \fn} = (h^+ h) \cdot \#((\cR_{\fm} \times \cR_{\fn/\fm})/\cU), \end{equation}
where $h=\#\Cl(F)$. 
Meanwhile, there is an exact sequence
\begin{equation} \label{e:qseq}
 1 \rightarrow (\cR_{\fm} \times \cR_{\fn/\fm})/\cU \rightarrow Q_{\fm, \fn} \rightarrow \Cl^+(F) \times \Cl(F) \rightarrow 1, 
 \end{equation}
 where the second nontrivial arrow is $([\fa], [\fb]) \mapsto ([\fa/\fb], [\fa])$.
From (\ref{e:qseq})  we deduce that $\#Q_{\fm, \fn}$ also equals the right side of (\ref{e:qmn}), completing the proof of the corollary.
\end{proof}

\subsection{Constant term map} \label{s:ctm}

If $k$ is even we define
\[  C_k = \bigoplus_{\cusps(\fn)} \C.\] 
Then (\ref{e:borel}) implies that we have a well-defined constant term map
 \begin{equation} \label{e:condef}
 \con_k \colon M_k(\fn) \longrightarrow C_k, \qquad f \mapsto (c_{\cA}(0, f))_{[\cA]}. 
 \end{equation}

For $k$ odd we must deal with the sign ambiguity in (\ref{e:borel}). 
For $k$ an odd integer, let
\[ \tilde{C}_k = \bigoplus_{\lambda \in \Cl^+(F)} \  \bigoplus_{\substack{\cA \in \Gamma_{1, \lambda}(\fn) \backslash \GL_2^+(F) \\ \cA \text{ admissible}}} \C. \]
 Endow $\tilde{C}_k$ with a right action of the upper triangular Borel $B$ by \[ [A] \cdot A' \mapsto [AA']\sign(\N d). \]
We let \[ C_k = H_0(B, \tilde{C}_k) = \tilde{C}_k/\langle c \cdot b - c\colon c \in \tilde{C}_k, b \in B \rangle. \]

Of course, Theorem~\ref{t:admissible} implies that
\[ C_k \cong \bigoplus_{\cusps^*(\fn)} \C. \]
However, fixing such an isomorphism requires making a non-canonical choice, which we would like to avoid. For positive odd $k$ we again have a canonical constant term map
\begin{equation} \label{e:condef2}
 \con_k \colon M_k(\fn) \longrightarrow C_k
 \end{equation}
that sends a modular form $f = (f_\lambda)$ to the tuple of normalized constant terms $c_{\cA}(0,f)$.  The discussion above implies that this map is well-defined.

\subsection{Cusps above $\infty$ and $0$}

We introduce the suggestive notation \[ C_\infty(\fn) =  \fQ_{\fn, \fn} =  \{ [\cA] \in \cusps(\fn) \colon \fn \mid \fc_\cA\}. \]
 This is the smallest set of cusps containing the cusps $\infty \in \Gamma_{1, \lambda}(\fn) \backslash \PP^1(F)$ for each $\lambda \in \Cl^+(F)$ that is stable under the action of the diamond operators $S(\fm)$.
It follows from Corollary~\ref{c:cuspsize} that $\#C_\infty(\fn) = h^+ h_\fn$, where $ h_\fn = \#G_\fn $, the size of the wide ray class group of conductor $\fn.$ 
Similarly, let
 \[ C_0(\fn) =  \fQ_{1, \fn} =  \{ [\cA] \in \cusps(\fn) \colon \gcd(\fn, \fc_\cA) = 1 \}. \]
 This is the smallest $G_\fn^+$-invariant set of cusps containing $0 \in \Gamma_{1, \lambda}(\fn) \backslash \PP^1(F)$ for each  $\lambda \in \Cl^+(F)$.

If $\fn = 1$, then $\cusps(\fn) = C_\infty(\fn) = C_0(\fn)$.  If $\fn$ is prime, then $\cusps(\fn)$ is a disjoint union of $C_\infty(\fn)$ and $C_0(\fn)$.  In general there will be more cusps.

If $\fb$ is a divisor of $\fn$, then for $* = 0, \infty$ we define $C_{*}(\fb, \fn)$ to be all elements in $\cusps(\fn)$ whose image under the canonical map $\cusps(\fn) \rightarrow \cusps(\fb)$ lies in $C_{*}(\fb)$, i.e.\ 
\[ 
C_{\infty}(\fb, \fn )  = \ \bigsqcup_{\fb \mid \fm \mid \fn}  \mathfrak{Q}_{\fm, \fn}, \qquad C_0(\fb, \fn) =  \ \bigsqcup_{\substack{\fm | \fn \\ (\fb, \fm) =1}} \mathfrak{Q}_{\fm, \fn}.
\]

\subsection{Lemma on level raising and cusps}

The following remark will be used in later computations.

\begin{lemma} \label{l:fq}
Let $f \in M_k(\fn)$ and let $\fq \nmid \fn$ be a prime ideal.  Let $\cA = (A, \lambda)$ represent a cusp
$[\cA] \in C_\infty(\fq\fm, \fq\fn)$ for some $\fm \mid \fn$.  There exists a pair $\cA' = (A', \mu)$ such that $[\cA'] \in C_\infty(\fm, \fn)$ and \begin{equation} \label{e:cfq}
 c_\cA(0,f|\fq) = c_{\cA'}(0, f) \end{equation}
\end{lemma}

\begin{proof}  Write $\cA=(A, \lambda)$.
By definition, 
\[ ((f|\fq)_\lambda)|_A = \N\fq^{-k/2} f_\mu|_{A'}, \quad \text{ where } A' = \mat{a_\mu \alpha}{a_\mu \beta}{\gamma}{\delta} \]
and $\fq \ft_\lambda = (a_\mu)\ft_\mu$. Let $\cA' = (A', \mu)$.  Then $[\cA'] \in C_\infty(\fm, \fn)$ and $c_\cA(0,f|\fq) = c_{\cA'}(0, f)$ are direct calculations using (\ref{e:lrdef}).
\end{proof}

\section{Eisenstein series} \label{s:eisenstein}

We recall the well-known Eisenstein series using \cite[\S2.2]{ddp} and \cite{al} as  convenient references. 
Let  $\eta, \chi$ be narrow ray class characters of modulus $\fa$, $\fb$, respectively, such that  $\eta\chi$ has sign $(k,k, \dotsc, k)$.
With the exception of the case $F=\Q, k=2, \eta^0 = \chi^0 = 1$, when there are convergence issues,
 there is an Eisenstein series $E_k(\eta, \chi) \in M_k(\fa\fb)$.   
Here $\eta$ and $\chi$ are not assumed to be primitive characters, and $\eta^0, \chi^0$ denote the primitive characters
 associated to $\eta, \chi$.  The Eisenstein series $E_k(\eta, \chi)$ has  $q$-expansion coefficients given by
\[ c(\fm, E_k(\eta, \chi)) = \sum_{\fr \mid \fm} \eta(\fm/\fr) \chi(\fr) \N\fr^{k-1}. \]
For $k > 1$, we have
\[  c_\lambda(0, E_k(\eta, \chi) ) = \begin{cases} 2^{-n} \eta^{-1}(\lambda) L(\chi \eta^{-1}, 1-k) & \text{ if } \fa = 1 \\
0 & \text{ if } \fa \neq 1.
\end{cases}
\]
For $k=1$, $\fa = 1$, and $\eta = 1$, we note that 
\[
c_\lambda(0, E_1(1, \chi)) = 2^{-n} \cdot \left\{ \begin{array}{cc} L(\chi, 0) & \text{ if } \fb \neq 1 \\
L(\chi, 0) + \chi^{-1}(\ft_{\lambda}) L(\chi^{-1}, 0) & \text{ if } \fb = 1
\end{array} \right.
\]
for all $\lambda \in \Cl^+(F)$.

Given a fixed level $\fn$, the Eisenstein subspace $E_k(\fn) \subset M_k(\fn)$ is defined to be subspace spanned by the Eisenstein series $E_k(\eta, \chi)|\fq$ where $\eta, \chi$ are primitive narrow ray class characters of conductor $\fa$, $\fb$, respectively, such that  $\eta\chi$ has sign $(k,k, \dotsc, k)$ and $\fa \fb \fq \mid \fn$.  An elementary argument using Hecke operators shows that for $k > 1$, these Eisenstein series are linearly independent (see \cite{al}*{Prop.\ 3.8}).  For $k=1$, we have $E_1(\eta, \chi)|\fq = E_1(\chi, \eta)|\fq$, and these equations generate the space of relations among the Eisenstein series.

The Eisenstein subspace is a complement to the space of cusp forms $S_k(\fn)$, i.e.\ we have \begin{equation}
 M_k(\fn) = E_k(\fn) \oplus S_k(\fn). \label{e:mes}
 \end{equation}
Furthermore, for $k \ge 2$ (excluding the case $F=\Q, k=2$) the restriction of the map $\con_k$ defined in (\ref{e:condef}) and (\ref{e:condef2}) to the subspace $E_k(\fn)$ is an isomorphism:
\begin{equation} \label{e:conke}
 \con_{k,E} \colon E_k(\fn) \xrightarrow{\sim} C_k. 
 \end{equation}
The results (\ref{e:mes}) and (\ref{e:conke}) are proven in \cite[Prop.\ 1.5]{wileseis} for weight $k=2$, and we sketch now a  proof in the general case.
Firstly, one can show  that (excluding the case $F = \Q, k=2$)  that there is an equality of dimensions in (\ref{e:conke}):
\begin{equation} \label{e:ecc}
 \dim_{\C} E_k(\fn) = \begin{cases}
 \#\cusps(\fn) = \sum_{\fm \mid \fn} \#Q_{\fm, \fn} \text{ if } k \ge 2 \text{ is even}, \\
\# \cusps^*(\fn) = \sum_{\fm \mid \fn}^* \#Q_{\fm, \fn}  \text{ if } k \ge 3 \text{ is odd}.  
 \end{cases} \end{equation}
To see this for $k$ even, note that $\#Q_{\fm, \fn}$ is the number of pairs $(\chi_1, \chi_2)$ where $\chi_1, \chi_2$ are ray class characters of  modulus $\fm, \fn/\fm$, respectively, such that $\chi_1 \chi_2$ is totally even. To such a pair we can associate the Eisenstein series $E_k(\chi_1^0, \chi_2^0)|_{\fm/\cond(\chi_1)}$.  Here 
$\chi_1^0$ and $\chi_2^0$ denote the primitive avatars of $\chi_1$ and $\chi_2$, respectively.  These Eisenstein series form the defining basis for $E_k(\fn)$.

For $k \ge 3$ odd, it is not hard to show that there exists a pair of characters $\chi_1, \chi_2$ of modulus $\fm, \fn/\fm$, respectively, with $\chi_1 \chi_2$ totally odd 
if and only if $(\fm, \fn)$ is an admissible pair.  
In this case, $\#Q_{\fm, \fn}$ is equal to the size of the set of such pairs $(\chi_1, \chi_2)$. 
The argument then continues as in the case for $k$ even, and we find that  the dimension of $E_k(\fn)$ is $\sum_{\fm \mid \fn}^* \#Q_{\fm, \fn}$.

With (\ref{e:ecc}) in hand, both (\ref{e:mes}) and (\ref{e:conke}) for $k \ge 2$ follow from the fact that \[ E_k(\fn) \cap S_k(\fn) = \{0\}, \] which is usually proven using the Petersson inner product (see for instance \cite{al}*{Prop.\ 3.9} or \cite{wileseis}*{Page 423}).

\subsection{Evaluation of constant terms of Eisenstein series}

Let $\chi_1$ and $\chi_2$ be ray class characters of modulus $\fa$ and $\fb$ and signatures $q_1$ and $q_2$, respectively. Put $\fn = \fa\fb$. Let $k$ be a positive integer and assume $q_1+q_2 \equiv (k,\ldots, k) \pmod{2}$. In this section we compute the constant terms of Eisenstein series $E_k(\chi_1, \chi_2)$ at various cusps $\cA = (A, \lambda)$, where $A = \begin{psmallmatrix}\alpha & *\\ \gamma & *  \end{psmallmatrix} \in \GL_2^+(F)$.

 We write $\fa_0 = \cond(\chi)$, $\fb_0=\cond(\psi)$ and let $\fa_1 = \fa/\fa_0$, $\fb_1 = \fb/\fb_0$.  Without loss of generality, we assume that $\gcd(\fa_0, \fa_1) = 1$ and that $\fa_1$ is square-free, since increasing the modulus at a prime already dividing the conductor or increasing the power of a prime already dividing the modulus does not affect the character or associated Eisenstein series.  We make the same assumptions about $\fb_1$.

\begin{definition} The Gauss sum associated to a primitive character $\chi$ of conductor $\fb$ and sign $r \in (\Z/2\Z)^n$ is given by
\[ \tau(\chi) = \sum_{x \in \fb^{-1}\fd^{-1} / \fd^{-1}} \!\!\!\!\!\!\!\! \sgn(x)^r \chi(x \fb \fd) e_F(x). \]
For a general character $\chi$ we define $\tau(\chi) = \tau(\chi^0)$ where $\chi^0$ is the associated primitive character.
\end{definition}

\medskip

Recall the invariants $\fa_\cA, \fb_\cA, \fc_\cA$ defined in \S\ref{ss:adm}.

\begin{definition} \label{d:pdef}
Assume that $[\cA] \in C_\infty(\fb, \fn)$, i.e.\ that $\fb \mid \fc_{\cA}$. Write $\fb_0 = \cond(\chi_2)$.
Let $\chi$ denote the  primitive character associated to $\chi_1\chi_2^{-1}$, and write $\ff = \cond(\chi)$.
Define
\begin{align}
& \!\!\!\!\! P_\cA(\chi_1, \chi_2, k) =  \nonumber \\ 
 & \frac{1}{2^n} \frac{\tau(\chi_1 \chi_2^{-1})}{\tau(\chi_2^{-1})} \left( \frac{\N\fb_0}{\N\ff}\right)^k \sgn(-\gamma)^{q_1} \sgn(\alpha)^{q_2} \chi_1(\fc_{\cA}/\fb_0) (\chi_2^0)^{-1}(\fa_{\cA})   L(\chi^{-1}, 1-k), \label{e:pdef}
\end{align}  
where $\chi_2^0$ denotes the primitive avatar of $\chi_2$. Here and throughout this article, we adopt the convention that $\chi_1(\fm) =0$ if $\gcd(\fm, \fa) \neq 1$, and similarly for any ray class character.   We also use the convention that if $a=0$ and $\fm$ is a fractional ideal, then 
\[ \sgn(a)^{q_1} \chi_1(a \fm)= \begin{cases}
0 & \text{ if }  \fa \neq 1, \\
\chi_1(\fm) & \text{ if } \fa = 1.
\end{cases}
\]  For example, when $\gamma = 0$ in (\ref{e:pdef}) the expression $\sgn(-\gamma)^{q_1} \chi_1(\fc_\cA/\fb)$ should be interpreted as 0 if $\fa \neq 1$ and as $\chi_1^{-1}(\ft_\lambda \fd \fb_A \fb)$ if $\fa =1$.
The analogous convention holds for the term $\sgn(\alpha)^{q_2} \chi_2^{-1}(\fa_{\cA}) = \sgn(\alpha)^{q_2} \chi_2^{-1}(\alpha \fb_\cA^{-1}).$
 
 For finite sets $S$ and $T$ of finite places of $F$, define
\[  P_\cA(\chi_1, \chi_2, k, S, T) =  P_\cA(\chi_1, \chi_2, k) \prod_{\fq \in S} (1 - \chi^{-1}(\fq)\N\fq^{k-1}) \prod_{\fq \in T} (1 - \chi(\fq) \N\fq^{-k}).
\]
Further, for an ideal $\fa \mid \fn$ we define
\[
\delta_{0, \cA}(\fa) = \left\{ \begin{array}{cc} 0 & \text{ if } [\cA] \notin C_0(\fa, \fn) \\
1 & [\cA] \in C_0(\fa, \fn) \end{array} \right.
\]
and
\[
\delta_{\infty, \cA}(\fa) = \left\{ \begin{array}{cc} 0 & \text{ if } [\cA] \notin C_\infty(\fa, \fn) \\
1 & [\cA] \in C_\infty(\fa, \fn). \end{array} \right.
\]
\end{definition}

\begin{remark}  Suppose that $\fa_0  = \cond(\chi_1)$ and $\fb_0 = \cond(\chi_2)$ are coprime.  Then
\[
 \tau(\chi_1\chi_2^{-1}) = \chi_1(\fb_0) \chi_2^{-1}(\fa_0) \tau(\chi_1) \tau(\chi_2^{-1}) \]  and hence
 \begin{equation} \label{e:reformp}
 P_\cA(\chi_1, \chi_2, k) = \frac{\tau(\chi_1)}{2^n (\N\fa_0^k)}  \sgn(-\gamma)^{q_1} \sgn(\alpha)^{q_2} \chi_1(\fc_{\cA}) \chi_2^{-1}(\fa_{\cA}/\fa_0)   L(\chi^{-1}, 1-k).
\end{equation}
\end{remark}

We require one more piece of notation.

\begin{definition}  Let $\fm \mid \fn$.  We write $J_\fm = J_\fm(\cA)$ for the set of prime divisors $\fq \mid \fm$ such that $[\cA] \in C_0(\fq, \fn)$ and $J_\fm^c$ for the set of prime divisors $\fq \mid \fm$ such that $[\cA] \in C_0(\fq, \fn)$.
\end{definition}

\begin{theorem} \label{t:eiscon1}
\begin{itemize}
\item[(1)] Let $k >  1$. Assume that $\chi_2$ is primitive of conductor $\fb$. The normalized constant term of $E_k(\chi_1, \chi_2)$ at $\cA$ equals
\begin{equation} \label{e:constterm}
\delta_{\infty, \cA}(\fb) P_{\cA}(\chi_1, \chi_2, k, \emptyset, T_{\fn, \ff}),
\end{equation}
where $T_{\fn,\ff}$ is the set of primes dividing $\fn$ but not $\ff$. 
\item[(2)] Let $k=1$. Assume that $\chi_2$ is primitive of conductor $\fb$. Suppose further that $\fa$ and $\fb$ are coprime. The  normalized constant term of $E_1(\chi_1, \chi_2)$ at a cusp $\cA$ equals
\begin{equation} \label{e:consttermwt1}
\begin{aligned}
 \delta_{0,\cA}(\fa) & \delta_{\infty, \cA}(\fb)P_{\cA}(\chi_1, \chi_2, 1, \emptyset, T_{\fn, \ff})  \\
+ \  \delta_{\infty,\cA}(\fa_0)& \delta_{0, \cA}(\fb)P_{\cA}(\chi_2, \chi_1, 1,J_{\fa_1}^c, \emptyset) \prod_{\fq \in J_{\fa_1}} (1 - \N\fq^{-1}),  
\end{aligned} 
\end{equation}

\end{itemize}
\end{theorem}

\begin{remark} \label{r:wt1}
Note that $E_1(\chi_1, \chi_2) = E_1(\chi_2, \chi_1)$.  The theorem only assumes that $\chi_2$ is primitive; if we assume also that $\chi_1$ is primitive, then in the setting of part (2) the sets $T_{\fn, \ff}$, $J_{\fa_1}$ and $J_{\fa_1}^c$ are empty and (\ref{e:consttermwt1}) becomes symmetric with respect to $\chi_1$ and $\chi_2$.
\end{remark}

\begin{proof}[Proof of Theorem~\ref{t:eiscon1}]   Recall the definition of the Eisenstein series $E_k(\chi_1, \chi_2)$ given in \cite[Prop. 3.4]{shim} (see also \cite[section 2.2]{ddp}). Let 
\[
U = \{u \in \cO_F^* : \N u^k = 1, u \equiv 1 (\text{mod }\fn)\}.
\]
For $k \geq 1$, we have $E_k(\chi_1, \chi_2) = (f_{\lambda})$, where $f_{\lambda}(z)$ is defined via Hecke's trick as follows. For $z \in \cH$ and $s \in \C$ with $\text{Re}(2s+k) > 2$, define 
\begin{equation} \label{eq:heckeeseries0}
f_{\lambda}(z,s) = C_{\lambda} \tau(\chi_2) \sum_{\fr \in \Cl(F)} \N\fr^k  g_\lambda(z, s) 
\end{equation}
where 
\[ 
C_{\lambda} = \frac{\sqrt{d_F}\Gamma(k)^n \N(\ft_{\lambda})^{-k/2}}{[\cO_F^*:U]\N(\fd)\N(\fb)(-2 \pi i)^{kn}}
\] 
and 
\begin{align} \label{eq:heckeeseries}
g_\lambda(z, s)  &=  \sum_{a,b} \frac{\sgn(a)^{q_1}\chi_1(a \fr^{-1})\sgn(-b)^{q_2} \chi_2^{-1}(-b\fb \fd \ft_{\lambda} \fr^{-1})}{(az+b)^k|az+b|^{2s}}  \\
\begin{split}
& = \sum_{(a_0,b_0)} \sgn(a_0)^{q_1}\chi_1(a_0 \fr^{-1})\sgn(-b_0)^{q_2} \chi_2^{-1}(-b_0\fb \fd \ft_{\lambda} \fr^{-1}) \ \times \label{e:heckeeseries2} \\
&  \hspace{4cm}  \sum_{(a,b) \equiv (a_0, b_0)} \frac{1}{(az+b)^k|az+b|^{2s}}.
 \end{split}
\end{align}

 The sum in (\ref{eq:heckeeseries0}) runs over representatives $\fr$ for the wide class group $\Cl(F)$.
 The sum in (\ref{eq:heckeeseries}) runs over representatives $(a,b)$ for the nonzero elements of the product $\fr \times \fd^{-1} \fb^{-1} \ft_{\lambda}^{-1} \fr$ modulo the diagonal action of $U$. In equation (\ref{e:heckeeseries2}) the sum $(a_0, b_0)$ runs through $(\fr/\fr\fa\fb) \times (\fd^{-1} \fb^{-1}\ft_{\lambda}^{-1} \fr/ \fd^{-1} \ft_{\lambda}^{-1} \fr \fa)$, while $(a,b)$ ranges over nonzero elements of $\fr \times \fd^{-1} \fb^{-1} \ft_{\lambda}^{-1} \fr$ modulo the diagonal action of $U$ such that $a \equiv a_0 \pmod{\fr\fa\fb}$ and $b \equiv b_0 \pmod{\fd^{-1} \ft_\lambda^{-1} \fr\fa}$.
 
  Here we use 
\[
\chi_1(a) = \sgn(a)^{q_1} \text{   for  } a \equiv 1 \pmod{\fa}, \qquad  \chi_2(b) = \sgn(b)^{q_2} \text{   for  } b \equiv 1 \pmod{\fb}.
\]
We remark that in the definition (\ref{eq:heckeeseries}) we already use that $\chi_2$ is primitive, applying \cite[equation (3.11)]{shim}).
The function $f_\lambda(z, s)$ can be analytically continued in the variable $s$ to the entire complex plane, and we set $f_\lambda(z) = f_\lambda(z, 0)$. 
 
We choose representatives of the cusp $[\cA] = [(A, \lambda)]$, with $A = \begin{psmallmatrix}\alpha & \beta \\ \gamma & \delta  \end{psmallmatrix}$, as follows. 
Let $\fg = \gcd(\fb, \fc_{\cA})$. The cusp $[\cA]$ only depends on $(\alpha, \gamma)$, so for convenience we are free to choose $\beta, \delta$ such that 
\begin{equation} \label{e:cuspcondition}
\det(A) = 1, \qquad \beta \in (\ft_{\lambda} \fd \fb_{\cA} \fg)^{-1}, \qquad \delta \in \fb(\fb_{\cA}\fg)^{-1}.
\end{equation}
Such $\beta, \delta$ exist by the definition of $\fg$. The map $(a,b) \mapsto (u, v) = (a,b)A$ induces a bijection 
\begin{equation} \label{e:Abijection}
\fr \times \fr(\fd\ft_{\lambda}\fb)^{-1} \longrightarrow \fr \fb_{\cA} \fg \fb^{-1} \times \fr (\fd \ft_{\lambda} \fb_{\cA} \fg)^{-1}.
\end{equation}
This bijection restricts to a bijection
\[
\fr \fa \fb \times \fr (\fd \ft_{\lambda})^{-1} \fa \longrightarrow \fr \fb_{\cA} \fg \fa \times \fr (\fd \ft_{\lambda} \fb_{\cA} \fg)^{-1} \fn.
\]
The function $g_{\lambda}(z,s)|_A$ can be written
\begin{align}
 g_{\lambda}(z,s)|_A 
= &\sum_{u_0, v_0} \sgn(u_0\delta - v_0\gamma)^{q_1} \chi_1((u_0\delta-v_0\gamma)\fr^{-1}) \sgn(u_0\beta - v_0 \alpha)^{q_2} \chi_2^{-1}((u_0 \beta - v_0 \alpha)\fb \fd \ft_{\lambda} \fr^{-1})  \nonumber \\
& \times \sum_{(u,v) \equiv (u_0, v_0)} \frac{1}{(uz+v)^k |uz+v|^{2s}}. \label{e:gdef}
\end{align}
Here $u_0$ and $v_0$ run through complete sets of representatives of \[ \fr\fb_{\cA} \fg \fb^{-1}/ \fr \fb_{\cA} \fg \fa \quad  \text{ and } \quad \fr(\fd \ft_{\lambda} \fb_{\cA} \fg)^{-1}/\fr(\fd \ft_{\lambda} \fb_{\cA} \fg)^{-1} \fn,\] respectively, and the pair $(u,v)$ runs through representatives for
$$
\left((\fr \fb_{\cA} \fg \fb^{-1} \times \fr( \fd \ft_{\lambda} \fb_{\cA} \fg)^{-1}) \setminus \{(0,0)\} \right)/U
$$  
such that $(u,v) \equiv (u_0, v_0)$.
We now recall notation from \cite[\S 3]{shim}. Up to a constant factor the  sum 
\[
\sum_{(u,v) \equiv (u_0, v_0)} \frac{1}{(uz+v)^k |uz+v|^{2s}}
\]
equals the series $E_{k,U}(z, u_0, v_0, \fr \fb_{\cA} \fg \fa, \fr (\fd \ft_{\lambda} \fb_{\cA} \fg)^{-1} \fn)$ defined in \cite[equation (3.1)]{shim}, with $r$ in \emph{loc.\ cit.}\ set to 0.

 We are now ready to prove (1), though we remark that much of what is said below also applies to (2). We have $k \geq 2$. By \cite[equation (3.7)]{shim} the constant term of $g_{\lambda}(z,s)|_A$ at $s=0$ is the value at $s=0$ of 
\begin{align*}
\sum_{u_0,v_0} & \sgn(u_0\delta - v_0\gamma)^{q_1} \chi_1((u_0\delta-v_0\gamma)\fr^{-1}) \sgn(u_0\beta - v_0 \alpha)^{q_2} \chi_2^{-1}((u_0 \beta - v_0 \alpha)\fb \fd \ft_{\lambda} \fr^{-1}) \\
& \times (-1)^{kn} \delta(u_0, \fr \fb_{\cA}\fg \fa) \sum_{v \equiv v_0} \sgn(\N v)^k |\N v|^{-k-2s}.
\end{align*}
Here $\delta(u_0, \fr \fb_{\cA}\fg \fa) = 0$ if $u_0 \notin \fr \fb_{\cA}\fg \fa$ and is 1 otherwise. Therefore the constant term of $g_{\lambda}(z,s)|_A$ at $s=0$ is zero if $u_0 \notin \fr \fb_{\cA}\fg \fa$. 

If $u_0 \in \fr \fb_{\cA}\fg \fa$, then using the relation $b_0 \gamma = u_0 - a_0 \alpha$ we deduce that $b_0\gamma \in \fr \fb_{\cA}$. On the other hand $b_0 \gamma \in  \fr \fb_{\cA} \fc_{\cA} \fb^{-1}$. Therefore 
\[
 b_0 \gamma \in  \fr \fb_{\cA} \cap  \fr \fb_{\cA} \fc_{\cA} \fb^{-1} = \fr \fb_{\cA} \fc_{\cA} \fg^{-1}.
\]
From this we obtain $b_0 \in \fr(\ft_\lambda \fd \fg)^{-1}.$
Now consider the case that $\cA$ does not represent a cusp in $C_{\infty}(\fb, \fn)$. This is equivalent to  $\fb \nmid \fc_{\cA}$ which in turn is equivalent to $\fg \neq \fb$. Hence $b_0 \fb \fd \ft_{\lambda} \fr^{-1}$ is an integral ideal not coprime to $\fb$. Therefore $\chi_2^{-1}(-b_0 \fb \fd \ft_{\lambda} \fr^{-1}) = 0$ and hence the constant term of $g_{\lambda}(z,s)|_A$ at $s=0$ is 0. Note that here we have used $b_0 = -u_0 \beta + v_0 \alpha$.

Next we turn to the case that $\cA$ does represent a cusp in $C_{\infty}(\fb, \fn)$, so $\fg = \fb$. As observed above we are only interested in the term with $u_0 \in \fr \fb_{\cA}\fg \fa$. We choose $u_0 = 0$ to represent the trivial coset  in $\fr \fb_{\cA}\fg \fb^{-1}/\fr \fb_{\cA}\fg \fa$. Therefore the constant term of $g_{\lambda}(z,s)|_A$ at $s=0$ is the value at $s=0$ of 
\begin{align*}
& \sum_{v_0} \sgn(-v_0)^{k} \sgn(\gamma)^{q_1} \sgn(\alpha)^{q_2} \chi_1(- v_0 \gamma \fr^{-1}) \chi_2^{-1}(-v_0 \alpha \fb \fd \ft_{\lambda} \fr^{-1}) (-1)^{kn} \sum_{v \equiv v_0} \frac{\sgn(\N v)^k}{|\N v|^{k+2s}}  \\
 & = \sum_{v} \sgn(\gamma)^{q_1} \sgn(\alpha)^{q_2} \chi_1(-v \gamma \fr^{-1}) \chi_2^{-1}(-v\alpha \fb \fd \ft_{\lambda} \fr^{-1}) \N(v\cO_F)^{-k-2s}.
\end{align*}
Hence the constant term of $f_{\lambda}(z,s)|_A$ at $s=0$ is the value of the following sum at $s=0$:
\begin{align*}
 &   C_{\lambda} \tau(\chi_2) \sum_{\fr \in \Cl(F)} \N\fr^k \sum_{v} \sgn(\gamma)^{q_1} \sgn(\alpha)^{q_2} \chi_1(-v \gamma \fr^{-1}) \chi_2^{-1}(-v\alpha \fb \fd \ft_{\lambda} \fr^{-1}) \N(v\cO_F)^{-k-2s} \\
= & \ C_{\lambda} \tau(\chi_2) \sgn(\gamma)^{q_1} \sgn(\alpha)^{q_2} \chi_1(\gamma(\fb \fd \ft_{\lambda} \fb_{\cA})^{-1}) \chi_2^{-1}(\alpha (\fb_{\cA})^{-1}) \N(\fb\fd \ft_{\lambda} \fb_{\cA})^{k+2s}  \\
& \ \  \times \sum_{\fr} \sum_{v} \chi_1(v \fb \fd \ft_{\lambda} \fb_{\cA} \fr^{-1}) \chi_2^{-1}(v \fb \fd \ft_{\lambda} \fb_{\cA} \fr^{-1}) \N(v \fb \fd \ft_{\lambda} \fb_{\cA} )^{-k-2s} \N(\fr^{-1})^{-k}.
\end{align*}
The value of this at $s=0$ is 
\begin{align*}
& C_{\lambda} \tau(\chi_2) \sgn(\gamma)^{q_1} \sgn(\alpha)^{q_2} \chi_1(\fc_{\cA}/\fb) \chi_2^{-1}(\fa_{\cA}) \N(\fb \fd \ft_{\lambda}\fb_{\cA})^{k} \times \\
& [\cO_F^*:U]L(\chi, k) \prod_{\fq \in T_{\fn, \ff}}(1 - \chi(\fq) \N\fq^{-k}),
\end{align*}
where $\chi$ is the primitive character associated to $\chi_1 \chi_2^{-1}$ and
 $\fq$ runs through the set $T_{\fn, \ff}$ of all primes dividing $\fn$ but not  $\ff = \cond(\chi).$ Next we use the functional equation 
\[
L(\chi, k) = \frac{d(F)^{\frac{1}{2}-k} N\ff^{1-k} (2 \pi i )^{kn}}{2^n \Gamma(k)^n \tau(\chi^{-1})} L(\chi^{-1}, 1-k)
\] 
together with the relations 
\begin{align*}
\tau(\chi_2)\tau(\chi_2^{-1}) &= \sgn(-1)^{q_2} \N\fb, \\ 
\tau(\chi)\tau(\chi^{-1}) &= \sgn(-1)^{q_1+q_2} \N\fn.
\end{align*}
We find that the unnormalized  constant term $a_{\lambda,A}(0)$ of $f_{\lambda}(z,0)|_A$ is 
\begin{align*}
 \frac{\tau(\chi_1 \chi_2^{-1})}{\tau(\chi_2^{-1})} \left( \frac{\N\fb\fb_{\cA}}{\N\ff}\right)^k  \N\ft_{\lambda}^{k/2}\sgn(-\gamma)^{q_1} \sgn(\alpha)^{q_2} & \chi_1(\fc_{\cA}/\fb) \chi_2^{-1}( \fa_{\cA})\ \times \\
  & \frac{L(\chi^{-1}, 1-k)}{2^n}  \prod_{\fq}(1 - \chi(\fq) \N\fq^{-k}).
\end{align*}
The normalized constant term is $(\N\fb_{\cA})^{-k}\N \ft_{\lambda}^{-k/2}$ multiplied by this, yielding statement (1) of the theorem.

Everything up to this point also applies when $k=1$. However, when $k=1$, the formula in \cite[equation (3.7)]{shim} shows that there is  an additional term which arises from the constant term in the $q$-expansion of \[ \sum_{(u,v) \equiv (u_0,v_0)} \frac{1}{(uz+v)|uz+v|^{2s}} \] at $s=0$; its value is the following sum at $s=0$:
\[
\frac{(-2\pi i )^{kn} \N(\fb^{-1})}{2^n\sqrt{d(F)}} \sum_{u \equiv u_0} \frac{\sgn \N(u)}{|\N u|^{2s}}.
\] 
Therefore the second term in the constant term of $g_{\lambda}(z,s)|_A$ at $s=0$ is the value of the following at $s=0$:
\begin{align}
& \sum_{u_0, v_0} \sgn(u_0\delta - v_0 \gamma)^{q_1} \chi_1((u_0 \delta - v_0 \gamma)\fr^{-1}) \sgn(u_0 \beta - v_0 \alpha)^{q_2} \chi_2^{-1}((u_0 \beta - v_0 \alpha) \fb \fd \ft_{\lambda} \fr^{-1}) \nonumber \\
& \ \ \ \  \times  \frac{(-2\pi i )^{kn} \N(\fb^{-1})}{2^n\sqrt{d(F)}} \sum_{u \equiv u_0} \frac{\sgn \N(u)}{|\N u|^{2s}}  \nonumber \\
 \begin{split}
 = & \  \frac{(-2\pi i )^{kn} \N(\fb^{-1})}{2^n\sqrt{d(F)}} \sum_u \frac{\sgn \N u}{|\N u|^{2s}}   \\
&  \ \ \ \  \times \sum_{v_0} \sgn(u\delta - v_0 \gamma)^{q_1} \chi_1((u \delta - v_0 \gamma)\fr^{-1}) \sgn(u \beta - v_0 \alpha)^{q_2} \chi_2^{-1}((u \beta - v_0 \alpha) \fb \fd \ft_{\lambda} \fr^{-1}). \label{e:lastsum}
\end{split}
\end{align}
Here the second sum in (\ref{e:lastsum}) runs through all $v_0$ in a  set of representatives for \[ \fr(\fd\ft_{\lambda} \fb_{\cA} \fg)^{-1}/\fr(\fd\ft_{\lambda} \fb_{\cA} \fg)^{-1} \fa\fb. \] By the definition of $\fg,$ we have $\gamma \in \fd \ft_\lambda \fb_{\cA} \fg$, and hence $v_0 \in \fr(\fd\ft_{\lambda} \fb_{\cA}\fg)^{-1}\fa \Rightarrow v_0 \gamma \in \fr\fa$. 
Therefore the last sum above (i.e.\ the expression appearing in (\ref{e:lastsum}) after the $\times$ symbol) can be written as a double sum 
\begin{align}
& \sum_{v_0 \in \fr(\fd\ft_{\lambda} \fb_{\cA} \fg)^{-1}/\fr(\fd\ft_{\lambda} \fb_{\cA} \fg)^{-1} \fa}   \sgn(u\delta - v_0 \gamma)^{q_1} \chi_1((u \delta - v_0 \gamma)\fr^{-1}) \label{e:vsum0}  \\
&  \ \ \ \  \ \ \ \ \ \ \ \times   \left( \sum_{\substack{v_0' \in \fr(\fd\ft_{\lambda} \fb_{\cA}\fg)^{-1}/\fr(\fd\ft_{\lambda} \fb_{\cA} \fg)^{-1} \fa\fb \\ v_0' \equiv v_0 \!\!\!\!\!\pmod{\fr(\fd\ft_{\lambda} \fb_{\cA} \fg)^{-1} \fa}}}   \sgn(u \beta - v_0' \alpha)^{q_2} \chi_2^{-1}((u \beta - v_0' \alpha) \fb \fd \ft_{\lambda} \fr^{-1}) \right).  \label{e:largesum}
\end{align}

Recall that the ``finite part" of the character $\chi_2$ is the character
\[ \chi_{2, f} \colon (\cO_F/\fb)^* \rightarrow \C^*, \qquad \chi_{2, f}(\alpha) = \sgn(\alpha)^{q_2} \chi_2((\alpha)). \]
We extend $\chi_{2, f}$ to a function of $\cO/\fb$ by dictating $\chi_{2,f}(\alpha) = 0$ if $\gcd(\alpha, \fb) \neq 1$.
Up to multiplication by a nonzero scalar,  the expression (\ref{e:largesum}) in large parenthesis is the sum of $\chi_{2,f}$ over a coset of the ideal  in $\cO/\fb$ generated by $\fa_{\cA} \fb \fg^{-1}$.  Since $\chi_{2}$ is primitive of conductor $\fb$, it is elementary that such a sum vanishes unless $\fa_{\cA} \fb \fg^{-1}$ is divisible by $\fb$, i.e.\ unless $\fg = \cO_F$.
 In other words, if  $[\cA] \not\in C_0(\fb, \fn)$ then the  sum (\ref{e:largesum}) is 0 and if $[\cA] \in  C_0(\fb, \fn)$ then (\ref{e:largesum}) equals \[ \N\fb \cdot \sgn(u \beta)^{q_2} \chi_2^{-1}(u \beta \fb \fd \ft_{\lambda} \fr^{-1}). \]

As we now show, a similar argument implies that the sum (\ref{e:lastsum}) is zero unless we also have  $[\cA] \in C_\infty(\fa_0, \fb)$. Since $\fg=1$, the sum 
\begin{equation} \label{e:vsum}
  \sum_{v_0} \sgn(u\delta - v_0 \gamma)^{q_1} \chi_1((u \delta - v_0 \gamma)\fr^{-1}). \end{equation}
appearing in (\ref{e:vsum0})
 is the sum of $\chi_{1, f}$ over a coset of the ideal in $\cO/\fa$ generated by $\gamma(\fd \ft_\lambda \fb_\cA)^{-1} = \fc_{\cA}$.  This vanishes unless $\fa_0 = \cond(\chi_1)$ divides $\fc_\cA$.  Hence (\ref{e:vsum}) vanishes unless  $\fa_0 \mid \fc_A$, i.e.\ unless $[\cA] \in C_\infty(\fa_0, \fb)$.  Furthermore when  $[\cA] \in C_\infty(\fa_0, \fb)$ the value of (\ref{e:vsum}) can be easily calculated directly.  Let $\fa_2 = \prod_{\fq \in J_{\fa_1}^c}\fq$ and $\fa_3 = \fa_1/\fa_2 = \prod_{\fq \in J_{\fa_1}} \fq$.   Then the value of (\ref{e:vsum}) is 
 \[
 \N\fa_3 \prod_{\fq \mid \fa_3} (1 - \N\fq^{-1}) \sgn(u_0\delta)^{q_1} \chi_1^{*}(u_0\delta\fr^{-1}),
\]
where $\chi_1^*$ is the character $\chi_1$ with modulus $\fa_0\fa_2$.  

Combining these calculations, we find that for $[\cA] \in C_0(\fb, \fn) \cap C_{\infty}(\fa_0, \fn)$, the second part of the constant term of $f_{\lambda}(z,s)|_A$ at $s=0$ is the value at $s=0$ of the following:
\begin{align*}
& C_{\lambda} \tau(\chi_2)  \frac{(-2\pi i )^{n} }{2^n\N(\fr (\fd \ft_{\lambda} \fb_{\cA})^{-1}\fa_3)\sqrt{d(F)}} \cdot \N\fa_3 \prod_{\fq \mid \fa_3} (1- \N\fq^{-1}) \\
& \times \sum_{\fr \in \Cl(F)} \N\fr \sum_{u_0} \sgn(u_0\delta)^{q_1} \sgn(u_0\beta)^{q_2} \chi_1^*(u_0\delta\fr^{-1}) \chi_2^{-1}(u_0\beta \fb \fd \ft_{\lambda} \fr^{-1}) \sum_{u \equiv u_0} \frac{\sgn \N u}{|\N u|^{2s}} \\
= \  &  C_{\lambda} \tau(\chi_2)  \frac{(-2\pi i )^{n}}{2^n\N((\fd \ft_{\lambda} \fb_{\cA})^{-1} )\sqrt{d(F)}} \prod_{\fq \mid \fa_3} (1- \N\fq^{-1}) \sgn(\delta)^{q_1} \sgn(\beta)^{q_2} \chi_1^*(\delta \fb^{-1} \fb_{\cA}) \chi_2^{-1}(\beta  \fd \ft_{\lambda} \fb_{\cA})  \\
& \times \sum_{\fr} \N\fr \sum_{u} \chi_1^*\chi_2^{-1}(u \fb (\fb_{\cA})^{-1}\fr^{-1}) \frac{1}{|\N u|^{2s}}.
\end{align*}
The value at $s=0$ is 
\begin{equation} \label{e:finally}
\frac{\tau(\chi_2) \N\ft_{\lambda}^{1/2} \N\fb_{\cA}}{2^n \N\fb} \sgn(\delta)^{q_1} \sgn(\beta)^{q_2} \chi_1^*(\delta \fb^{-1} \fb_{\cA}) \chi_2^{-1}(\beta  \fd \ft_{\lambda} \fb_{\cA}) L(\chi, 0) \prod_{\fq \mid \fa_2} (1 - \chi(\fq)) \prod_{\fq \mid \fa_3}(1 - \N\fq^{-1}).
\end{equation}
Since $\fa_0 \fa_2 \mid \fc_\cA$, it follows that $\beta \gamma \in \fa_0 \fa_2$ and hence $\alpha \delta \equiv 1 \pmod{\fa_0\fa_2}$, whence
\[  \sgn(\delta)^{q_1} \chi_1^*(\delta \fb^{-1} \fb_{\cA}) =  \sgn(\alpha)^{q_1} (\chi_1^*)^{-1}(\fa_\cA/\fb) = \sgn(\alpha)^{q_1} \chi_1^{-1}(\fa_\cA/\fb), \]
where the last equality follows since $\gcd(\fa_{\cA}, \fa_2)=1$.  Similarly $\alpha \delta \in \fb \Rightarrow -\beta\gamma \equiv 1 \pmod{\fb}$, hence 
\[   \sgn(\beta)^{q_2} \chi_2^{-1}(\beta  \fd \ft_{\lambda} \fb_{\cA}) = \sgn(-\gamma)^{q_2} \chi_2(\fc_\cA).
\]
Therefore, noting (\ref{e:reformp}), after scaling by the normalization factor $(\N\fb_{\cA})^{-1} \N\ft_{\lambda}^{-1/2}$ for constant terms, the value in (\ref{e:finally}) is equal to \[ P_\cA(\chi_2, \chi_1, 1)\prod_{\fq \mid \fa_2} (1 - \chi(\fq)) \prod_{\fq \mid \fa_3}(1 - \N\fq^{-1}). \]

The first term calculated above (for $k \geq 1$) is non-zero only when the cusp $[\cA]$ belongs to $C_{\infty}(\fb, \fn) \cap C_0(\fa, \fn)$. The second term is non-zero only when $[\cA]$ belongs to $C_0(\fb, \fn) \cap C_{\infty}(\fa_0, \fn)$.
This finishes the proof.
\end{proof}

\subsection{Constant terms for raised level and imprimitive characters} \label{s:rl}

In our arithmetic application \cite{dk}, we require the constant terms of the level-raised Eisenstein series $E_k(\chi, \psi)|_\fm$ for  auxiliary squarefree ideals $\fm$, with $\chi$ and $\psi$ possibly imprimitive.  This level raising is related to the $T$-smoothing operation of Deligne--Ribet \cite{dr}.

The following notation will be in effect throughout this section.
 Let $\chi$ and $\psi$ be characters of modulus $\fa$ and $\fb$ and signatures $q_1$ and $q_2$, respectively. Let $k$ be a positive integer such that $q_1+q_2 \equiv (k, \ldots, k) \pmod{2}$. We denote the conductors of $\chi$ and $\psi$ by $\fa_0$ and $\fb_0$, respectively and put $\fa_1 = \fa/\fa_0$ and $\fb_1 = \fb/\fb_0$. Assume $\gcd(\fb_1, \fa) = 1$. Let $\fn = \fa\fb\fl$ for a square-free integral ideal $\fl$ with $\gcd(\fa\fb,\fl) = 1$.   We assume that $\fa_1$ is squarefree and coprime to $\fa_0$, and similarly for $(\fb_1, \fb_0)$.

Let $\cA = (A, \lambda)$ with \[ A = \left( \begin{array}{cc} \alpha & * \\ \gamma & * \end{array}\right) \in \GL_2^+(F), \qquad \lambda \in \Cl^+(F).\]

\begin{theorem}\label{t:levelraised} Let $\fm$ be a divisor of $\fl$. The normalized constant term of $E_k(\chi, \psi)|_{\fm}$ at $\cA$ is given as follows:
\begin{itemize}
\item If $k \geq 2$, then the normalized constant term of $E_k(\chi, \psi)|_{\fm}$ at $\cA$ is
\begin{equation}
\delta_{\infty, \cA}(\fb_0) P_{\cA}(\chi, \psi, k, J_{\fb_1}^c, J_{\fa_1}) \prod_{\fq \in J_{\fb_1}}(1 - \N\fq^{-1}) \prod_{\fq \in J_{\fm}}(\psi(\fq) \N \fq^k)^{-1} \prod_{\fq \in J_{\fm}^c}\chi^{-1}(\fq).
\end{equation}
\item If $k=1$, we further assume that $\gcd(\fa, \fb) =1$. Then the normalized constant term of $E_1(\chi, \psi)|_{\fm}$ at $\cA$ is
\label{e:delta}
\begin{equation}
\begin{aligned}
& \delta_{0, \cA}(\fa) \delta_{\infty, \cA}(\fb_0) P_\cA(\chi, \psi, 1, J_{\fb_1}^c, J_{\fa_1}) \prod_{\fq \in J_{\fb_1}}(1 - \N\fq^{-1}) \prod_{\fq \in J_{\fm}} (\psi(\fq)\N\fq)^{-1} \prod_{\fq \in J_{\fm}^c} \chi^{-1}(\fq) \\
+ \ & \delta_{\infty, \cA}(\fa_0) \delta_{0, \cA}(\fb) P_\cA(\psi, \chi, 1, J_{\fa_1}^c, J_{\fb_1}) \prod_{\fq \in J_{\fa_1}}(1 - \N\fq^{-1}) \prod_{\fq \in J_{\fm}} (\chi(\fq)\N\fq)^{-1} \prod_{\fq \in J_{\fm}^c} \psi^{-1}(\fq) 
\end{aligned}
\label{e:delta1}
\end{equation}
\end{itemize}
\end{theorem}

\begin{remark}  The term $\delta_{0, \cA}(\fa)$ is unnecessary in (\ref{e:delta1}) since $P_\cA(\chi, \psi, 1)$ already vanishes if $[\cA] \not\in C_0(\fa, \fn)$.  We include this factor simply as a reminder that this portion of the constant term is supported on $C_0(\fa, \fn) \cap C_\infty(\fb_0, \fn)$. 
\end{remark}

\begin{proof}  We give the  proof for $k \geq 2$.  The argument for $k=1$ is identical and left to the reader.

First we assume that $\fb = \fb_0$, i.e. $\fb_1 = 1$ and calculate the constant term of $E_k(\chi, \psi)|_{\fm}$. Let $\fm = \fq_1 \cdots \fq_j$ and use induction on $j$. The base case $j=0$ follows directly from (\ref{e:constterm}). For  $j>0$ we use the expression 
\begin{equation} \label{e:level}
E_k(\chi_{\fm}, \psi) = \sum_{\ft \mid \fm} \mu(\ft) \chi(\ft)E_k(\chi, \psi)|_{\ft}.
\end{equation}
Here $\chi_\fm$ denotes the character $\chi$ viewed with modulus $\fa\fm$. If $J_\fm = \{\fq_1, \ldots,\fq_j\}$, then $\delta_{0, \cA}(\fa\fm) = \delta_{0, \cA}(\fa)$, so by (\ref{e:constterm}) the normalized constant term of $E_k(\chi_\fm, \psi)$ at $\cA$ is
\begin{equation}\label{e:el}
\delta_{0, \cA}(\fa) \delta_{\infty, \cA}(\fb_0) P_{\cA}(\chi, \psi, k, \emptyset, J_{\fa_1}) \prod_{\fq \mid \fm}(1- \chi\psi^{-1}(\fq_i)\N \fq_i^{-k}). 
\end{equation}
The induction hypothesis gives the normalized constant term at $\cA$ of each term on the right side of (\ref{e:level}) except for $E_k(\chi, \psi)|_{\fm}$.  Therefore one can use (\ref{e:level}) and (\ref{e:el}) to solve for  the normalized constant term of 
$E_k(\chi, \psi)|_{\fm}$ at $\cA$.  One  obtains  
\begin{align*}
\delta_{0, \cA}(\fa) \delta_{\infty, \cA}(\fb_0) P_{\cA}(\chi, \psi, k, \emptyset, J_{\fa_1}) \prod_{\fq \mid \fm}(\chi\psi^{-1}(\fq_i)\N \fq_i^{k})^{-1}
\end{align*}
as desired.  Now suppose  $J_{\fm} \neq \{\fq_1, \ldots, \fq_j\}$, which is equivalent to $[\cA] \notin C_0(\fm, \fn)$. Then $\delta_{0, \cA}(\fa\fm) =0$ so (\ref{e:constterm}) implies that the constant term of $E_k(\chi_{\fm}, \psi)$ at $\cA$ is 0. Without loss of generality, assume that $\fq_j \notin J_\fm$. For every  subset $I \subset \{\fq_1, \ldots, \fq_{j-1}\}$, put $\mathfrak{q}_I = \prod_{\fq_i \in I} \mathfrak{q}_i$. If $I \neq \{\fq_1, \dotsc, \fq_{j-1}\},$ then we can apply the induction hypothesis to the forms $ E_k(\chi, \psi)|_{\ft}$ for both $\ft = \fq_I$ and $\ft = \fq_I \fq_j$ on the right side of (\ref{e:level}) to see that the contributions made by their constant terms at $\cA$ cancel. It follows that the constant term of $E_k(\chi, \psi)|_{\fm}$ at $\cA$ equals $\chi^{-1}({\fq_j})$ times that of $E_k(\chi, \psi)|_{\fm/\mathfrak{q}_j}$, and we are done by the induction hypothesis. 

Next we relax the condition that $\fb_1 =1$.
 We use the expression 
\begin{equation} \label{e:b1gen}
E_k(\chi, \psi)|_{\fm} = \sum_{\ft \mid \mathfrak{b}_1} \mu(\ft) \psi(\ft) (\N\ft)^{k-1}E_k(\chi, \psi^0)|_{\ft\fm},
\end{equation}
where $\psi^0$ is the primitive character associated with $\psi$. The case already completed for $\psi$ primitive gives the constant terms of the forms on the right of (\ref{e:b1gen}).  The result then follows from the formula
\begin{align*}
\sum_{\ft \mid \fb_1} & \mu(\ft) \psi(\ft) (\N\ft)^{k-1} \prod_{\fq \in J_{\ft}}(\psi(\fq) \N\fq^k)^{-1} \prod_{\fq \in J_{\ft}^c} \chi^{-1}(\fq) \\ 
= & \prod_{\fq \in J_{\fb_1}}(1 - \N\fq^{-1}) \prod_{\fq \in J_{\fb_1}^c}(1- \chi^{-1} \psi(\fq) \N\fq^{k-1}).
\end{align*}
\end{proof}

\section{Ordinary forms} \label{s:ordinary}

Let $\fp$ be a prime ideal of $\cO_F$ dividing a prime number $p$.  Following Hida, we define the ordinary operator
\[ e_\fp^{\ord} = \lim_{n \rightarrow \infty} U_\fp^{n!}.\]
Let $\mathfrak{P} = \gcd(p^\infty, \fn)$ be the $p$-part of $\mathfrak{n}$.  We define
\[ e_\fP^{\ord} = \prod_{\fp \mid \fP} e_{\fp}. \]
Let $E$ be a finite extension of $\Q_p$. The space of $\fP$-ordinary forms is defined by:
\[ M_k(\fn, E)^{\fP\dord} = e_\fP^{\ord}  M_k(\fn, E) \]
This is the largest subspace on which the operator $U_\fp$ acts invertibly  for each $\fp \mid \fP$.

\begin{theorem} \label{t:cusp}  A form $f \in M_k(\mathfrak{n}, E)^{\fP\dord}$ is cuspidal if and only if its constant terms at all cusps in $C_{\infty}(\mathfrak{P}, \mathfrak{n})$ are zero. If $f \in M_k(\mathfrak{n}, E)$ has constant terms zero at all cusps in $C_\infty(\fP, \fn)$, then $e_\fP^{\ord}(f)$ is cuspidal.
\end{theorem}

We provide two proofs of Theorem~\ref{t:cusp}.  The first proof is longer, but its method could have other applications, so we include full details.  We begin with the following elementary lemma from linear algebra.

\begin{lemma}  \label{l:jordan}
Let $V$ be a finite dimensional vector space over a field and let  $B = \{v_1, \dotsc, v_n\}$ be a basis.
Let $S$ be a possibly infinite set of commuting endomorphisms of $V$ satisfying the following properties:
\begin{itemize}
\item After re-ordering, the matrix for each $T \in S$ with respect to the basis $B$ is in Jordan canonical form.
\item Every Jordan block of size greater than 1 has associated eigenvalue 0.
\end{itemize}
Let $B' \subset B$ be the set of basis vectors that are actual (non-generalized) eigenvectors for every $T \in S$.  Suppose the elements of $B'$ are distinguished by their $S$-eigenvalues, i.e.\ for  $v_i \neq v_j$ in $B'$, there exists $T \in S$ such that the $T$-eigenvalues of $v_i$ and $v_j$ are distinct.
Finally let $W \subset V$ be a subpace that is preserved by each $T$.  Then $W$ is nonzero if and only if it contains some $v_i \in B'$. 
\end{lemma}

\begin{proof}  Suppose $v = \sum a_i v_i \in W$ with the $a_i$ not all zero.  We first show that we can find another nonzero $v' \in W$
such that its expression as a linear combination of elements in $B$ only  contains elements of $B'$.  For this, suppose that $v_i \in B \setminus B'$ occurs in $v$ with a nonzero coefficient $a_i$.  Let $T \in S$ such that $v_i$ is not an eigenvector for $T$.  Then there is a unique $n \ge 1$ such that $T^n(v_i) \in B$ is an eigenvector for $T$.   We replace $v$ by $T^n(v)$.  This is another element of $W$; its expression as a linear combination of the $v_i$ has at most as many elements of $B \setminus B'$ as did $v$.
And the term $a_i v_i$ has been replaced by $a_i T^n(v_i)$---this uses the fact that the $T$-eigenvalue of $T^n(v_i)$ is 0.
  Note in particular that since $T^n(v_i) \in B$ occurs with a nonzero coefficient, $T^n(v) \neq 0$.  If $T^n(v_i) \in B'$, we have reduced the number of elements of $B \setminus B'$ in our linear combination.  If not, there is some other $T' \in S$ that we can apply a certain number of times, say $m$, to replace $T^n(v_i)$ by its associated $T'$-eigenvector.  Continuing in this way, we get a sequence of nonzero vectors $v \rightarrow T^n(v) \rightarrow (T')^mT^n(v) \rightarrow \cdots $ and a corresponding sequence of terms occuring in the expression of these vectors in terms of  $B$: 
  \[ a_i v_i \rightarrow a_i T^n(v_i) \rightarrow a_i (T')^mT^n(v) \rightarrow \cdots.\]
  Since this latter sequence clearly cannot a cycle, and $B$ is finite, it must terminate. This occurs when the corresponding element of $B$ actually lies in $B'$.  We have therefore created a new nonzero element of $W$ whose expression in the basis $B$ contains fewer elements of $B \setminus B'$.  Continuing this procedure yields a nonzero element of $W$ that is in the span of $B'$.
  
  Now let $v = \sum a_i v_i \in W$ with $v_i \in B'$ be such an element.  If more than one $v_i$ occurs in this linear combination with nonzero coefficient, say $v_i$ and $v_j$, then by assumption we can find  $T \in S$ such that the associated eigenvalues $\lambda_i(T)$ and $\lambda_j(T)$ are distinct.  We can replace $v$ by $T(v) - \lambda_1(T) v$.  This annihilates the $v_i$ term, but is nonzero because it does not annihilate the $v_2$ term.  Furthermore it has fewer nonzero coefficients than $v$.  Continuing in this way, we can repeatedly decrease the number of terms in the expression of $v$ until we find that there is some $v_i \in B' \cap W$.
\end{proof}

\begin{proof}[Proof 1 of Theorem~\ref{t:cusp}]  The second statement of the theorem follows from the first since $e_\fP^{\ord}(f)$ preserves the space $E[C_{\infty}(\mathfrak{P}, \mathfrak{n})].$  To prove the first statement, let $f \in M_k(\fn, E)^{\fP\dord}$ be a form whose constant terms at all cusps in $C_{\infty}(\fP, \fn)$ are zero.
Then $f$ is a sum of a cusp form and a linear combination of Eisenstein series.   The cusp form does not affect any constant terms; we can therefore assume that $f$ is a linear combination of Eisenstein series, and we must show that $f=0$.  The Eisenstein subspace has the following convenient basis, for which each of the Hecke operators is in Jordan canonical form:
\begin{equation} \label{e:jordan}
 B = \{ E_k(\eta_{\fr}, \psi_{\fs})|_\fc \}, \text{ where}
 \end{equation}
\begin{itemize}
\item $\eta$ and $\psi$ are primitive characters of conductor $\fa$, $\fb$ respectively.
\item $\fr, \fs$ are each squarefree products of primes such that $\gcd(\fa, \fr) = \gcd(\fb, \fs) = 1$.
\item $\fa\fb\fr\fs$ is divisible by all primes dividing $\fn$.
\item $\fc$ is only divisible by primes dividing $\gcd(\fa\fr, \fb\fs)$.
\item $\fa\fb\fr\fs\fc$ divides $\fn$.
\end{itemize}
Since this is a lot of notation, it is behooves us to demonstrate this with an example.  Suppose that $\eta$ and $\psi$ are primitive of conductor $\fa, \fb$ with associated Eisenstein series $E_k(\eta, \psi) \in M_k(\fa\fb)$.  Let $\fp$ be a prime not dividing $\fa\fb$.
For $n  \ge 1$, the generalized eigenspace of $M_k(\fa\fb \fp^n)$ corresponding to $E_k(\eta, \psi)$---i.e.\ the subspace on which all the Hecke operators away from $\fp$ act via the eigenvalues of $E_k(\eta, \psi)$---has 2 or 3 Jordan blocks for the action of $U_\fp$: (1) the form $E_k(\eta_\fp, \psi)$ with $U_\fp$-eigenvalue $\psi(\fp)\N\fp^{k-1}$, (2) the form $E_k(\eta, \psi_\fp)$  with $U_\fp$-eigenvalue $\eta(\fp)$, and (3) if $n \ge 2$, a Jordan block with  $U_\fp$-eigenvalue 0 and basis $E_k(\eta_\fp, \psi_\fp)|_{\fp^i}$ as $i=0, \dotsc, n-2$.
Here $U_\fp(E_k(\eta_\fp, \psi_\fp)|_{\fp^i}) = E_k(\eta_\fp, \psi_\fp)|_{\fp^{i-1}}$ for $i \ge 1$, and  $U_\fp(E_k(\eta_\fp, \psi_\fp)) = 0$.
  The basis (\ref{e:jordan}) is the generalization of this case to the general setting.

Now, the space $E_k(\fn, E)^{\fP\dord}$ is the subspace of $E_k(\fn, E)$ generated by the subset $B_\fP \subset B$ consisting of the $E_k(\eta_{\fr}, \psi_{\fs})|_\fc$ such that $\fa\fr$ is coprime to $\fP$.   We apply Lemma~\ref{l:jordan} where the set of endomorphisms $S$ is the set of Hecke operators indexed by the primes not dividing $\fP$: the $T_\fq$ for $\fq \nmid \fn$ and $U_\fq$ for $\fq \mid \fn/\fP$.  The subspace $W \subset  
E_k(\fn, E)^{\fP\dord}$ is taken to be the subspace of elements whose constant terms at all cusps in $C_\infty(\fP, \fn)$ vanish.  This subspace is fixed by the Hecke operators away from $\fP$.
We need to prove that $W = \{0\}$, and the lemma implies that is suffices to show that no eigenvector in $B_\fP$ lies in $W$.  The subset $B_\fP' \subset B_\fP$ of eigenvectors is the set of $E_k(\eta_{\fr}, \psi_{\fs})|_\fc$ such that $\fa\fr$ is coprime to $\fP$ and $\fc = 1$.
It remains to prove that for such a form $E_k(\eta_{\fr}, \psi_{\fs})$, there exists a cusp $[\cA] \in C_\infty(\fP, \fn)$ such that the constant term  at $\cA$ is nonzero. 

For this, we first note that it suffices to show this at the minimal level at which $E_k(\eta_{\fr}, \psi_{\fs})$ appears, namely $\fn' = \fa\fb\fr\fs$.  Indeed, if we let $\fP'$ be the $p$-part of $\fn'$, then the canonical map $C_\infty(\fP, \fn) \rightarrow C_\infty(\fP', \fn') $ is surjective.  Therefore we assume that $\fn = \fa\fb\fr\fs$.  

Write $\fs = \fs_1 \fs_2$, where $\fs_1 = \gcd(\fs, \fa\fr)$. Note that $\fP \mid \fb \fs_2$. Let $\cA$ be a cusp in $C_{\infty}(\fb\fs, \fn) \subset C_{\infty}(\fP, \fn)$ such that $\fc_{\cA}/\fb\fs_1$ is coprime with $\fa\fr$. We use the expression
\[
E_k(\eta_\fr, \psi_\fs) = \sum_{\fm \mid \fs_1} \mu(\fm) \psi(\fm)\N \fm^{k-1} E_k(\eta_\fr, \psi_{\fs_2})|\fm.
\]
to show that the constant term of $E_k(\eta_\fr, \psi_\fs)$ at $\cA$ is nonzero.
By Lemma \ref{l:fq}, the constant term of $E_k(\eta_{\fr}, \psi_{\fs_2})|_{\fm}$ at $\cA$ equals the constant term of $E_k(\eta_\fr, \psi_{\fs_2})$ at some other $\cA'$, where $[\cA'] \in C_\infty(\fb \fs /\fm, \fn/\fm)$.  If $\fm \neq \fs_1$, then by definition $\fb \fs /\fm$ is not coprime to $\fa \fr$.  Then $P_{\cA'}(\eta, \psi, k) = 0 $ because of the $\eta(\fc_{\cA'})$ factor, and hence by Theorem \ref{t:levelraised} the constant term is 0. On the other hand if $\fm = \fs_1$ then Theorem \ref{t:levelraised} shows that this constant term is nonzero as long as $\fc_{\cA'}/\fb$ is coprime with $\fa\fr$. This holds because $\fc_{\cA}/\fb\fs_1$ is coprime with $\fa\fr$. The result follows.
\end{proof}

\begin{proof}[Proof 2 of Theorem~\ref{t:cusp}] Our second proof of the first statement is a direct computation using  the action of the Hecke operator $U_{\fp}$ for each $\fp \mid \fP$.  First we note that since we are on the ordinary subspace, each  $U_\fp$ acts semisimply (see \cite{hida}*{pg.\ 382}).  Furthermore the operator $U_\fp$ preserves $C_\infty(\fP, \fn)$.   Therefore it suffices to consider the case where $f$ is a $U_\fp$-eigenvector for each $\fp \mid \fP$.  

Let us recall the explicit definition of the operator $U_\fp$. For each $\mu \in \Cl^+(F)$, let $\lambda \in \Cl^+(F)$ denote the class of $\mu \fp^{-1}$.  Write $\ft_\lambda \ft_\mu^{-1} \fp = (x)$ where $x$ is a totally positive element of $F^*$.
Given $\beta \in F$ define $m_\beta = \mat{1}{\beta}{0}{x}.$
Then $f|_{U_\fp} = (g_\mu)_{\mu \in \Cl^+(F)}$ where 
\begin{equation} \label{e:updef}
g_\mu = \N\fp^{(k-2)/2} \sum_{\beta \in  \ft_\mu^{-1} \fd^{-1}/  \ft_\mu^{-1} \fd^{-1}\fp} (f_\lambda)|_{m_\beta}. \end{equation}
Let $\mathcal{A} = (A, \mu)$ represent a cusp.  If $\fp^r \mid \fc_\cA$ with $r > 0$, then one readily checks that $\fp^{r+1} \mid \fc_{\cA'}$, where $\cA' = (m_\beta A, \lambda)$ is an associated cusp appearing in (\ref{e:updef}).  Therefore since $f$ is a $U_\fp$-eigenform with nonzero eigenvalue for each $\fp \mid \fP$, and its constant terms vanish on $C_\infty(\fP, \fn)$, then by applying $U_\fp$ repeatedly we see that the constant terms of $f$ vanish on $C_\infty(\fP_0, \fn)$, where $\fP_0$ is the product of the distinct primes dividing $\fP$.

Next, to show that $f$ has vanishing constant terms at all cusps in $\cusps(\fn)  = C_\infty(1, \fn)$, we show that we can remove the primes in $\fP_0$ one-by-one.  Therefore, let $\fP_1 \mid \fP_0$, and let $\fp \mid \fP_1$.  We will show that the cuspidality of $f$ on $C_\infty(\fP_1, \fn)$ implies its cuspidality on $C_\infty(\fP_1/\fp, \fn)$.  Sequentially removing all the primes $\fp \mid \fP_0$ in this fashion will  then give the desired result.

For this, we use the expression (\ref{e:updef}) once again.  We also introduce the notation $f(\cA)$ to denote the normalized constant term of $f$ at the cusp $\cA$.  If $\cA \in C_\infty(\fP_1/\fp, \fn)$ but $\cA \not \in C_\infty(\fP_1, \fn)$, then one can check directly from the definitions that there is a unique $\beta \in   \ft_\mu^{-1} \fd^{-1}/  \ft_\mu^{-1} \fd^{-1}\fp$ such that the associated cusp $\cA' = (m_\beta A, \lambda)$ also does not lie in $C_\infty(\fP_1, \fn)$; for all the other $\beta$, the associated cusp {\em does} lie in  $C_\infty(\fP_1, \fn)$.  The cuspidality of $f$ on $C_\infty(\fP_1, \fn)$ therefore implies that 
\[
a_\fp f(\cA) = f|_{U_{\fp}}(\cA) = \N \fp^{k-1} f(\cA'),
\]
where $a_\fp$ denotes the $U_\fp$-eigenvalue of $f$.  Note that the constant $\N \fp^{k-1}$ arises from tracing through our normalization factors on constant terms.
Now, the set $C_\infty(\fP_1/\fp, \fn) \setminus C_\infty(\fP_1, \fn)$ is finite, so continually repeating this process, the sequence \begin{equation} \label{e:aseq}
 \cA \rightarrow \cA' \rightarrow \cdots \end{equation}  must eventually arive at a repetition.  At this point we obtain an equation of the form
\[
a_\fp^r f(\cA'') = \N\fp^{(k-1)r} f(\cA'')
\]
for some positive integer $r$ and some cusp $\cA''$.
As the Hecke eigenvalue $a_\fp$ is a $p$-adic unit and $k > 1$,  we have  $a_\fp^r  \neq \N\fp^{(k-1)r}$ for any positive integer $r$.
We obtain $f(\cA'') = 0$, and hence the same is true for every other cusp appearing in the sequence (\ref{e:aseq}); in particular $f(\cA) = 0$ as desired. 
\end{proof}

\subsection*{Acknowledgements} We thank Jesse Silliman for helpful conversations during the writing of this paper.

\end{document}